\documentclass[a4paper,12pt]{article} 
\newcommand{\rmistyle}[1]{}
\newcommand{\ifrmirejected}[1]{#1}
\rmistyle{
\usepackage[journal=RMI,lang=american]{ems-journal-rmi}
\title{Geometry on principal curvature surfaces}
\titlemark{Geometry on principal curvature surfaces}
\emsauthor{1}{
	\givenname{Luciana}
	\surname{de F\'{a}tima Martins}
	\zblid{martins.luciana-f}
	\mrid{842358}
	\orcid{0000-0002-1265-6168}}{L.~Martins}
\emsauthor{2}{
	\givenname{Kentaro}
	\surname{Saji}
	\mrid{701192}
	\zblid{saji.kentaro}
	\orcid{0000-0002-1863-9713}}{K.~Saji}
\emsauthor{3}{
	\givenname{Samuel}
	\surname{Paulino dos Santos}
	\mrid{1344794}
	\zblid{dos-santos.samuel-paulino}
	\orcid{}}{S.~Santos}
\emsauthor{4}{
	\givenname{Runa}
	\surname{Shimada}
	\mrid{1688069}
	\zblid{runa.shimada}
	\orcid{0009-0006-7989-8846}}{R.~Shimada}
\Emsaffil{1}{
	\pretext{}
	\department{Departamento de Matem{\'a}tica}
	\organisation{IBILCE}
	\rorid{UNESP}
	\address{R. Crist{\'o}v{\~a}o Colombo}
	\zip{2265}
	\city{S{\~a}o Jos{\'e} do Rio Preto, SP}
	\country{Brazil}
	\posttext{CEP 15054-000}
	\affemail{luciana.martins@unesp.br}
	\furtheremail{}}
\Emsaffil{2}{
	\pretext{}
	\department{Department of Mathematics}
	\organisation{Graduate School of Science}
	\rorid{Kobe University}
	\address{Rokkodai}
	\zip{1-1}
	\city{Nada, Kobe}
	\country{Japan}
	\posttext{657-8501}
	\affemail{saji@math.kobe-u.ac.jp}
	\furtheremail{}}
\Emsaffil{3}{
	\pretext{}
	\department{Departamento de Matem{\'a}tica}
	\organisation{IBILCE}
	\rorid{UNESP}
	\address{R. Crist{\'o}v{\~a}o Colombo}
	\zip{2265}
	\city{S{\~a}o Jos{\'e} do Rio Preto, SP}
	\country{Brazil}
	\posttext{CEP 15054-000}
	\affemail{spsantos@uem.br}
	\furtheremail{}}
\Emsaffil{4}{
	\pretext{}
	\department{Department of Mathematics}
	\organisation{Graduate School of Science}
	\rorid{Kobe University}
	\address{Rokkodai}
	\zip{1-1}
	\city{Nada, Kobe}
	\country{Japan}
	\posttext{657-8501}
	\affemail{231s010s@stu.kobe-u.ac.jp}
	\furtheremail{}}
\classification[53A05]{57R45}

\usepackage{ascmac}
\usepackage{mathrsfs}
\usepackage[active]{srcltx}
\usepackage{amscd}
\usepackage{cases}
\usepackage[noadjust]{cite}
 \newtheorem{theorem}{Theorem}[section]
 \newtheorem{proposition}[theorem]{Proposition}
 
 \newtheorem{lemma}[theorem]{Lemma}
 \newtheorem{corollary}[theorem]{Corollary}
\theoremstyle{definition}

\numberwithin{equation}{section}
\numberwithin{figure}{section}
\renewcommand{\theenumi}{{\rm(\arabic{enumi})}}
\renewcommand{\labelenumi}{\theenumi}
\newcounter{introthm}
\theoremstyle{plain}
\newtheorem{introtheoremT}{Theorem}
\newenvironment{introtheorem}
  {\refstepcounter{introthm}%
   \renewcommand{\theintrotheoremT}{\Alph{introthm}}%
   \begin{introtheoremT}}
  {\end{introtheoremT}}
\newtheorem{introcorollaryT}{Corollary}

\newcommand{\red}[1]{\textcolor{red}{#1}}
\newcommand{\pur}[1]{\textcolor{purple}{#1}}
\newcommand{\blue}[1]{\textcolor{blue}{#1}}
\newcommand{\green}[1]{\textcolor{green}{#1}}
\newcommand{\R}{\boldsymbol{R}}
\newcommand{\C}{\boldsymbol{C}}
\newcommand{\Z}{\boldsymbol{Z}}
\newcommand{\N}{\boldsymbol{N}}
\newcommand{\A}{\mathcal{A}}
\newcommand{\ZZ}{\mathcal{Z}}
\newcommand{\pmt}[1]{{\begin{pmatrix} #1  \end{pmatrix}}}
\newcommand{\vect}[1]{\boldsymbol{#1}}
\newcommand{\bv}{\boldsymbol{V}}
\newcommand{\bw}{\boldsymbol{w}}
\newcommand{\inner}[2]{{#1}\cdot{#2}}
\newcommand{\spann}[1]{\left\langle{#1}\right\rangle}
\newcommand{\herm}{\operatorname{Herm}}
\newcommand{\hess}{\operatorname{Hess}}
\newcommand{\sgn}{\operatorname{sgn}}
\newcommand{\trace}{\operatorname{trace}}
\newcommand{\tr}{\operatorname{trace}}
\newcommand{\rank}{\operatorname{rank}}
\newcommand{\im}{\operatorname{Im}}
\newcommand{\first}{{\rm I}}
\newcommand{\second}{{\rm I\!I}}
\newcommand{\firstmat}{I}
\newcommand{\secondmat}{I\!I}
\renewcommand{\phi}{\varphi}
\renewcommand{\hom}{\operatorname{Hom}}
\renewcommand{\Gamma}{\varGamma}
\newcommand{\ep}{\varepsilon}
\newcommand{\RR}{{\mathcal R}}
\newcommand{\M}{{\mathscr M}}
\newcommand{\E}{{\mathscr E}}
\newcommand{\D}{{\mathscr D}}
\newcommand{\K}{{\mathcal K}}
\newcommand{\coef}{\operatorname{coef}}
\pagestyle{plain}
\renewcommand{\thefootnote}{\fnsymbol{footnote}}
\newcommand{\bgamma}{\vect{\gamma}}
\newcommand{\bxi}{\vect{\xi}}
\newcommand{\bnu}{\vect{\nu}}
\newcommand{\n}{\vect{n}}
\newcommand{\bn}{\vect{n}}
\newcommand{\be}{\vect{e}}
\newcommand{\e}{\vect{e}}
\newcommand{\bb}{\vect{b}}
\newcommand{\Gr}{\operatorname{Gr}}
\newcommand{\trans}[1]{{\vphantom{#1}}^t{\!#1}}
\begin{document}
\date{\today}
\begin{abstract}
	We study the geometry of the principal curvature surface associated with a Whitney umbrella. This surface is obtained by extending the Whitney umbrella in the normal direction by its bounded principal curvature and admits a natural geometric interpretation. We prove that its intersection with the normal plane coincides with the pedal curve of the curvature parabola and deduce that the focal conic is the inversion of this pedal curve.
	We then investigate the geometry of the principal curvature surface along the exceptional set by  studying its geodesic and normal curvatures, as well as its Gaussian and mean curvatures, and determining all possible generic numbers of zeros of these curvature functions.
\end{abstract}
\maketitle
\begin{ack}
Dedicated to the memory of Maria Aparecida Soares Ruas.
\end{ack}

\begin{funding}
The second author was partially supported by 
JSPS KAKENHI 25K07001, 22KK0034.
\end{funding}
} 
\ifrmirejected{
\setlength{\oddsidemargin}{2mm}
\setlength{\evensidemargin}{2mm}
\setlength{\topmargin}{-23mm}
\setlength{\textwidth}{156mm}
\setlength{\textheight}{258mm}
\renewcommand{\baselinestretch}{1}
\usepackage{times}
\usepackage{amsmath,amssymb}
\usepackage{amsthm}
\usepackage{ascmac}
\usepackage[dvipdfmx]{graphicx}
\usepackage{mathrsfs}
\usepackage[active]{srcltx}
\usepackage{amscd}
\usepackage{cases}
\usepackage[noadjust]{cite}
 \newtheorem{theorem}{Theorem}[section]
 \newtheorem{proposition}[theorem]{Proposition}
 \newtheorem{fact}[theorem]{Fact}
 \newtheorem{lemma}[theorem]{Lemma}
 \newtheorem{corollary}[theorem]{Corollary}
\theoremstyle{definition}
 \newtheorem{definition}[theorem]{Definition}
 \newtheorem{remark}[theorem]{Remark}
 \newtheorem{example}[theorem]{Example}
\numberwithin{equation}{section}
\numberwithin{figure}{section}
\renewcommand{\theenumi}{{\rm(\arabic{enumi})}}
\renewcommand{\labelenumi}{\theenumi}
\newcounter{introthm}
\theoremstyle{plain}
\newtheorem{introtheoremT}{Theorem}
\newenvironment{introtheorem}
  {\refstepcounter{introthm}%
   \renewcommand{\theintrotheoremT}{\Alph{introthm}}%
   \begin{introtheoremT}}
  {\end{introtheoremT}}
\newtheorem{introcorollaryT}{Corollary}
\newenvironment{introcorollary}
  {\refstepcounter{introthm}%
   \renewcommand{\theintrocorollaryT}{\Alph{introthm}}%
   \begin{introcorollaryT}}
  {\end{introcorollaryT}}
\usepackage[usenames]{color}
\newcommand{\red}[1]{\textcolor{Red}{#1}}
\newcommand{\pur}[1]{\textcolor{Purple}{#1}}
\newcommand{\blue}[1]{\textcolor{Blue}{#1}}
\newcommand{\green}[1]{\textcolor{Green}{#1}}
\newcommand{\R}{\boldsymbol{R}}
\newcommand{\C}{\boldsymbol{C}}
\newcommand{\Z}{\boldsymbol{Z}}
\newcommand{\N}{\boldsymbol{N}}
\newcommand{\A}{\mathcal{A}}
\newcommand{\ZZ}{\mathcal{Z}}
\newcommand{\pmt}[1]{{\begin{pmatrix} #1  \end{pmatrix}}}
\newcommand{\vect}[1]{\boldsymbol{#1}}
\newcommand{\bv}{\boldsymbol{V}}
\newcommand{\bw}{\boldsymbol{w}}
\newcommand{\inner}[2]{{#1}\cdot{#2}}
\newcommand{\spann}[1]{\left\langle{#1}\right\rangle}
\newcommand{\herm}{\operatorname{Herm}}
\newcommand{\hess}{\operatorname{Hess}}
\newcommand{\sgn}{\operatorname{sgn}}
\newcommand{\trace}{\operatorname{trace}}
\newcommand{\tr}{\operatorname{trace}}
\newcommand{\rank}{\operatorname{rank}}
\newcommand{\im}{\operatorname{Im}}
\newcommand{\first}{{\rm I}}
\newcommand{\second}{{\rm I\!I}}
\newcommand{\firstmat}{I}
\newcommand{\secondmat}{I\!I}
\renewcommand{\phi}{\varphi}
\renewcommand{\hom}{\operatorname{Hom}}
\renewcommand{\Gamma}{\varGamma}
\newcommand{\ep}{\varepsilon}
\newcommand{\RR}{{\mathcal R}}
\newcommand{\M}{{\mathscr M}}
\newcommand{\E}{{\mathscr E}}
\newcommand{\D}{{\mathscr D}}
\newcommand{\K}{{\mathcal K}}
\newcommand{\coef}{\operatorname{coef}}
\pagestyle{plain}
\renewcommand{\thefootnote}{\fnsymbol{footnote}}
\newcommand{\bgamma}{\vect{\gamma}}
\newcommand{\bxi}{\vect{\xi}}
\newcommand{\bnu}{\vect{\nu}}
\newcommand{\n}{\vect{n}}
\newcommand{\bn}{\vect{n}}
\newcommand{\be}{\vect{e}}
\newcommand{\e}{\vect{e}}
\newcommand{\bb}{\vect{b}}
\newcommand{\Gr}{\operatorname{Gr}}
\newcommand{\trans}[1]{{\vphantom{#1}}^t{\!#1}}
\begin{document}
\title{Geometry on principal curvature surfaces}
\author{
Luciana F. Martins,
Kentaro Saji,
Samuel P. dos Santos\\ and\\
Runa Shimada}
\date{\today}
\maketitle

\begin{abstract}
	We study the geometry of the principal curvature surface associated with a Whitney umbrella. This surface is obtained by extending the Whitney umbrella in the normal direction by its bounded principal curvature and admits a natural geometric interpretation. We prove that its intersection with the normal plane coincides with the pedal curve of the curvature parabola and deduce that the focal conic is the inversion of this pedal curve.
	We then investigate the geometry of the principal curvature surface along the exceptional set by  studying its geodesic and normal curvatures, as well as its Gaussian and mean curvatures, and determining all possible generic numbers of zeros of these curvature functions.
\end{abstract}
}
\section{Introduction}
A classical construction in differential geometry associates to a regular surface its focal surfaces, obtained by extending the surface in the normal direction by the reciprocals of its principal curvatures. In this paper, we investigate the geometry of the surfaces obtained by extending the surface in the normal direction by the principal curvatures themselves, which we call \emph{principal curvature surfaces}. We show that these surfaces also admit a natural geometric interpretation.

A remarkable feature of principal curvature surfaces is their relation with the curvature parabola. If one applies the definition of the curvature parabola for singular surfaces introduced by the first author and  Nu\~no-Ballesteros in \cite{mn} to a regular surface, then the curvature parabola degenerates into the line segment on the normal line whose endpoints correspond to the two principal curvatures. Consequently, the two principal curvature surfaces are precisely the endpoints of this segment. For a corank one singular surface, the curvature parabola is a plane curve in the normal plane that encodes the second-order geometry of the singularity and generalizes the curvature ellipse of regular surfaces. In the case of a Whitney umbrella, it is a genuine parabola.

This observation naturally raises the question of whether such a geometric interpretation persists for singular surfaces. Among singular surfaces in $\R^3$, Whitney umbrellas provide an ideal testing ground. They are stable singularities and possess a rich differential geometry, including well-defined principal curvatures, a curvature parabola, and a direction-dependent unit normal vector. Consequently, Whitney umbrellas  provide a natural and nontrivial setting in which the geometry of principal curvature surfaces can be investigated.

A map-germ
$f:(\R^2,0)\to(\R^3,0)$
is a Whitney umbrella if it is $\mathcal A$-equivalent to
$(u,v)\mapsto (u,v^2,uv)$
at the origin. Here, two map-germs $f_1, f_2:(\R^2,0)\to(\R^3,0)$
are said to be {\it  $\A$-equivalent} if there exist  germs of diffeomorphisms
$\lambda:(\R^2,0)\to(\R^2,0) $ and
$\mu:(\R^3,0)\to(\R^3,0) $ such that $\mu \circ f_1 \circ \lambda^{-1}= f_2$.  Moreover,
we denote by $S(f)$ the set of
singular points of $f$. For a Whitney umbrella
$f:(\R^2,0)\to(\R^3,0)$,
a direction-dependent unit normal vector can be defined via the blow-up of the source space \cite{fh-fronts}.  In this setting, only one principal curvature extends smoothly to the singular point, and hence there exists a unique principal curvature surface associated with the Whitney umbrella.

Let   $f:(\R^2,0)\to(\R^3,0)$ be a Whitney umbrella.  Let $\M$ denote the blow-up of the source plane at the origin, which records the incident directions at the singular point, and let $\M_0\subset\M$ be the open subset consisting of directions for which the radial vector does not belong to the kernel of $df$
(see Section \ref{sec:priwu} for details).
On $\M_0$, one of the principal curvatures extends smoothly, while the other diverges.
As a result, there is a unique principal curvature surface associated with the Whitney umbrella near the singularity.
We define
$$
h_{\kappa_1}=f+\kappa_1\nu:\M_0\to\R^3,
$$
where $\kappa_1$ is the smoothly extending principal curvature and
$\nu$ is a unit normal vector.
 Let $l^\perp=\big(df_0(T_0\R^2)\big)^\perp$ be the normal plane and  $\gamma:\R\to l^\perp$ be the curvature parabola of $f$. Let $P_{(\gamma,0)}$ be the pedal curve of $\gamma$ in $l^\perp$ with respect to the origin $0$ of $l^\perp\subset\R^3$.

Our first result shows that the principal curvature surface admits a natural extension of the geometric interpretation valid for regular surfaces: the intersection of the principal curvature surface with the normal plane coincides with the pedal curve of the curvature parabola. In this sense, the curvature parabola unifies the regular and Whitney umbrella cases, and the principal curvature surfaces are naturally encoded by this object in both settings. We formulate this result precisely in the following theorem.

\begin{introtheorem}\label{thm:pricipalpedal}
	Let\/ $f:(\R^2,0)\to(\R^3,0)$ be a Whitney umbrella,
	and let\/ $h_{\kappa_1}$ be the principal 
curvature surface
	associated with the principal curvature
	that extends smoothly to\/ $\M_0$.
	Then the intersection of the image 
of\/ $h_{\kappa_1}$ with
	the normal plane\/ $l^\perp$ coincides 
with the image of the pedal curve of
	the curvature parabola\/ $\gamma$ with 
respect to the origin:
	$$
	h_{\kappa_1}(\M_0)\cap l^\perp=P_{(\gamma,0)}(\R).
	$$
\end{introtheorem}

 The {\it focal set\/} of a map
$f:(\R^2,0)\to(\R^3,0)$ is defined as the set of points in $\R^3$
for which the distance squared function has a degenerate critical point.
For a regular surface, this coincides with the classical focal set.
For a Whitney umbrella, it lies in the normal plane
and is a conic \cite[Section 3.1]{fh-fronts}.
It is called the {\it focal conic\/} and is denoted by $FC(f)$.

 Our second result reveals an unexpected relationship between two fundamental invariants of a Whitney umbrella: the focal conic is obtained as the inversion of the pedal curve of the curvature parabola.

\begin{introtheorem}\label{cor:cocalinv}
	Let\/ $f:(\R^2,0)\to(\R^3,0)$ be a Whitney umbrella.
	Then the inversion in the unit circle of the image of the pedal curve of the curvature parabola with respect to the origin
	coincides with the focal conic of\/ $f$:
	$$
	i\big(P_{(\gamma,0)}(\R)\big)=FC(f).
	$$
\end{introtheorem}

Here, the {\it inversion in the unit circle\/}, or {\it inversion\/} for short,
is the map $i:\R^2\setminus \{0\}\to \R^2\setminus \{0\}$ 
defined by $i(x_1,x_2)=(x_1,x_2)/(x_1^2+x_2^2)$, and 
the inversion on $l^\perp$ is defined by 
identifying $(l^\perp\subset\R^3,0)$ with $(\R^2,0)$.
Note that this inversion differs from the complex $z\mapsto 1/z$.

The intersection of the principal curvature surface with the normal plane is the image of the exceptional set
$\E=\{(0,\theta)\in\M\}$
under $h_{\kappa_1}$.
The parameter $\theta$ of $\E$ corresponds to the direction from which points approach the Whitney umbrella singularity.
Hence, the geometry of the principal curvature surface along $\E$ reflects the geometry of the 
Whitney umbrella in the corresponding direction.

Motivated by this observation, we investigate the geodesic curvature and normal curvature of the exceptional set as a curve on the principal curvature surface, as well as the Gaussian and mean curvatures of the principal curvature surface along $\E$. For each of these functions we determine all possible  generic number of zeros. Here, the term \emph{generic} means that the corresponding property holds outside a proper algebraic subset of the space of Whitney umbrellas up to finite jet. Equivalently, these algebraic subsets are characterized 
by the vanishing of certain discriminants of polynomial 
expressions  and  by the vanishing of denominator and leading coefficients.
We note that the complement of an algebraic subset is open and dense.


The results of this paper reveal new relationships between principal curvature surfaces and the fundamental invariants of Whitney umbrellas, namely the curvature parabola and the focal conic. They also show that principal curvature surfaces provide a natural framework for studying curvature phenomena associated with singular surfaces.

Finally, the paper contains an appendix devoted to principal curvature surfaces associated with regular surfaces. To the best of our knowledge, this is the first systematic study of the singularities of principal curvature surfaces. In contrast with focal surfaces, principal curvature surfaces are not compatible with homothetic scaling, leading to a significantly different geometric behaviour.

The paper is organized as follows. In Section~2 we introduce principal curvature surfaces and prove their relation with the curvature parabola and the focal conic. Section~3 studies the geometry and singularities of principal curvature surfaces associated with Whitney umbrellas. Section~4 presents an explicit example. In Section~5 we investigate generic properties of the exceptional set and determine the number of zeros of several curvature functions. The appendix is devoted to the study of singularities of principal curvature surfaces associated with regular surfaces.


\section{Principal curvature surfaces}


\subsection{The curvature parabola and the principal curvature surface}
\label{sec:cpara}
Following \cite[Section 2.1]{mn}, we introduce
the curvature parabola.
Let $f:(\R^2,0)\to(\R^3,0)$ be a map-germ.
For a vector $X=a\partial_u+b\partial_v$,
we set $\first(X,X)=a^2E+2ab F+b^2G$, where
$E=f_u\cdot f_u$,
$F=f_u\cdot f_v$,
$G=f_v\cdot f_v$ are the coefficients of the
first fundamental form for a coordinate system $(u,v)$.
Furthermore, we set $\second(X,X)=\pi_l(a^2f_{uu}+2abf_{uv}+b^2f_{vv})$,
where $\pi_l:\R^3\to l^\perp$ is the orthogonal projection,
and $l=df_0(T_0\R^2)$ is the image of
the differential of $f$. 
If $\rank df_0=2$, then $l$ is the tangent plane, and hence
$l^\perp$ is the {\it normal line}.
If $\rank df_0=1$, then $l$ is a line,
and $l^\perp$ is the {\it normal plane}.
We set 
\begin{equation}\label{eq:secondxx}
\Delta=\{\second(X,X)|\first(X,X)=1,\ a,b\in\R\}.
\end{equation}
If $f:(\R^2,0)\to(\R^3,0)$ is a regular surface,
and $\nu:(\R^2,0)\to\R^3$ is its unit normal vector,
then $\Delta$ is a line segment 
$\{f+k\nu\,|\, \kappa_1\leq k\leq \kappa_2\}$,
where $\kappa_i$ $(i=1,2)$ are the principal curvatures
of $f$ satisfying $\kappa_1\leq \kappa_2$.
We assume $0$ is not an umbilic point.
Then $\kappa_i$ $(i=1,2)$ are $C^\infty$-functions.
We set 
$$
h_{\kappa_i}=f+\kappa_i\nu\quad(i=1,2),
$$
the maps that give the endpoints of the line segment mentioned above.
We call them the 
{\it principal curvature surfaces with respect to\/}
$\kappa_i$ $(i=1,2)$.
When $f$ has a singular point at $0$, 
if a unit normal vector $\nu$ along $f$ is well defined,
then one can define principal curvature surface in a natural way.
In fact, principal curvatures are defined as the solutions
of the equation
\begin{equation}\label{eq:preq}
\det\big(\secondmat-k\firstmat\big)=0,\quad
\left(
\firstmat=\pmt{E&F\\ F&G},\quad
\secondmat=\pmt{L&M\\ M&N}\right)
\end{equation}
in $k$,
where 
$L=f_{uu}\cdot\nu$,
$M=f_{uv}\cdot\nu$,
$N=f_{vv}\cdot\nu$ are the coefficients of the second
fundamental form,
provided that the solution $k$ is of class $C^\infty$.
A map-germ $f:(\R^n,0)\to (\R^{n+1},0)$ is called
a {\it frontal\/} if there exists
a map-germ $\nu:(\R^n,0)\to\R^{n+1}$ such that 
$df(X)\cdot\nu=0$ and $|\nu|=1$ hold for any 
$p\in(\R^n,0)$ and $X \in T_p\R^n$.
The map $\nu$ is called a {\it unit normal vector field\/} 
({\it unit normal\/} for short).
The map-germ $f$ is called a {\it front\/} if the map  $(f,\nu)$ is an immersion.



\subsection{Principal curvature surface of Whitney umbrella}
\label{sec:priwu}
Let $f:(\R^2,0)\to(\R^3,0)$ be a map-germ satisfying 
$\rank df_0=1$.
Then there exists a pair $(\xi,\eta)$ of vector fields 
such that they are linearly independent and 
$\eta$ generates the kernel of $df_0$.  
We call such a pair an {\it adapted\/} pair.
The map-germ $f$ is said to be $S$-{\it type\/}
(or {\it $SB$-type\/}) if
$f_\xi\times f_{\eta\eta}\ne0$ holds
for an adapted pair $(\xi,\eta)$.
Here, the directional derivative of a function $g$ along 
a vector field $\zeta$, written $\zeta g$, is denoted by $g_{\zeta}$. We note that $\zeta_2(\zeta_1g)=g_{\zeta_1\zeta_2}$.
For the meaning of $S$ and $B$,
see a classification of germs given in \cite{mond}.
In the classification, the series 
$S_k^\pm$-singularities $(k\geq0)$
forms its main part,
and all these singularities are of $S$-type.
Since a map-germ $f$ is of $S$-type if and only if
the two-jet of $f$ satisfies
$j^2 f(0,0)=(u,v^2,0)$ or $(u,v^2,uv)$,
a Whitney umbrella is of $S$-type.
The normal form of a Whitney umbrella is given in \cite{west};
see also \cite{bw}.
There are several studies on the geometry of  Whitney umbrellas
using this form, 
see for example, \cite{bw,fh-fronts,hhnuy,shimadas1,tari}.
A modified normal form that allows one to treat $S$-type germs
uniformly is given in \cite[Theorem 2.3]{shimadas1},
 where $S$-type germs
include all germs of codimension at most three in $C^\infty(2,3)$.
A similar normal form for a map-germ with 2-jet $(u,v^2,0)$
can also be found in \cite[Proposition 2.1]{fh-normal}.

\begin{proposition}{{\rm ($s=0$ case of \cite[Theorem 2.3]{shimadas1})}}
Let\/ $f:(\R^2,0) \to (\R^3,0)$ be an\/ $S$-type map-germ.
Then there exist an orientation preserving diffeomorphism-germ\/ 
$\phi : (\R^2, 0)$ $\to (\R^2, 0)$ and\/
$T \in SO(3)$ and\/ 
$f_{21}, f_{31}, f_{321}, f_{331} \in C^\infty(1,1)$, 
$f_{322} \in C^\infty(2,1)$ such that
\begin{eqnarray}
&&T \circ f \circ \phi^{-1}(u,v)\nonumber\\
&=&\big(u,u^2f_{21}(u)+v^2,
u^2f_{31}(u)+v^2(uf_{321}(u)+vf_{322}(u,v))+uv f_{331}(u)\big).
\label{eq:wunormal}
\end{eqnarray}\nonumber
If\/ $f$ is a Whitney umbrella, then\/
$f_{331}(0)>0$.
\end{proposition}

Let $f:(\R^2,0) \to (\R^3,0)$ be a Whitney umbrella.
Then the above normal form, namely, the functions
$f_{21}$, $f_{31}$, $f_{321}$, $f_{331}$, $f_{322}$,
are uniquely determined by $f$.
Hence, the coefficients of the functions appearing in \eqref{eq:wunormal}
are geometric invariants.
We set
$W_{21}(f)$,
$W_{31}(f)$,
$W_{331}(f)\in\R$ to be
the coefficients
$f_{21}(0)$,
$f_{31}(0)$,
$f_{331}(0)$
when $f$ is written in the form
\eqref{eq:wunormal}.

If $f$ is a Whitney umbrella, then
the set $\{\second(X,X)|\first(X,X)=1,\ a,b\in\R\}$ at $0$ is a non-degenerate parabola
(\cite[Theorem 2.5]{mn}),
called the {\it curvature parabola}.
To study the principal curvature surface, we discuss
the normal vector and
the principal curvatures near a Whitney umbrella.
In \cite{fh-fronts}, it is shown that
via the blow-up, a Whitney umbrella can be
regarded as a front.
To describe these properties within the framework of this paper,
we examine them in our context.

Let $f:(\R^2,0) \to (\R^3,0)$ be a Whitney umbrella.
Let $S^1=\R/2\pi\Z$ be the circle, and
let $\sim$ be an equivalence relation
on the set $\R\times S^1$ defined by
$(r,\theta)\sim(r',\theta')$ if
they are equal, or $(r',\theta')=(-r,\theta+\pi)$.
We set $\M$ to be the quotient space, and $\pi:\M\to\R^2$ to be the
natural map $(r,\theta)\mapsto(r\cos\theta,r\sin\theta)$,
which is usually called the blow-up.
Namely, $\M$ is
a space that records the incident directions at the origin.
We set 
$$
\M_0=\{(r,\theta)\in \M\,|\,\text{ if }r=0,\text{ then }
\partial_r
\not\in \ker df\text{ at }(0,\theta)\},
$$
which is an open set in $\M$.
We identify $f(u,v)$ with $f\circ\pi(r,\theta)$ and
write 
\begin{equation}\label{eq:bl}
f(r\cos\theta,r\sin\theta)=f\circ\pi(r,\theta).
\end{equation}

We call the set $\E=\{(0,\theta)\in\M\}$ the
{\it exceptional set}. As shown in \cite[Section 2]{fh-fronts}, 
if $f(u,v)$ is a Whitney umbrella, then
a unit normal vector can be defined as a map from $\M$. Namely, 
$\nu(r\cos\theta,r\sin\theta):(\M,\E)\to\R^3$, 
a map-germ along $\E$, is 
well-defined.
In fact, if $f$ is written in the form \eqref{eq:wunormal},
then $r$ can be factored out from $f_v$.
We set $\phi=f_v(r\cos\theta,r\sin\theta)/r$.
Then $f_u|_{r=0}=(1,0,0)$ and 
$\phi|_{r=0}=(0,2 \sin\theta,\cos\theta f_{331}(0))$. 
Since $f_u\times\phi\neq0$,
it determines a non-zero normal vector.
	Thus we obtain the well-defined unit normal vector
$\nu(r\cos\theta,r\sin\theta)$.
We remark that 
$$
\big\{\langle\nu|_{r=0}\rangle_{\R}\,\big|\,\theta\in S^1\big\}
=
(1,0,0)^\perp,
$$
which is the normal plane of $df_0(T_0\R^2)$.
	Moreover, as pointed out in \cite[Lemma 2.2]{fh-fronts},
one principal curvature extends smoothly to $r=0$,
whereas the other is unbounded there.
More precisely, we prove the following lemma. We set\/
$\E_0=\{(0,\theta)\in\E\,|\,\cos\theta\ne0\}$ and remark that $\E_0$ is closed in $\M_0$.

\begin{lemma}\label{lem:k1cinf}
The bounded principal curvature\/
$\kappa_1:(\M_0,\E_0)\to\R$
extends smoothly to\/ $\E_0$.
In particular, $\kappa_1$ is a\/ $C^\infty$-function.
Moreover, the  associated principal curvature surface $$h_{\kappa_1}(r\cos\theta,r\sin\theta)
=f(r\cos\theta,r\sin\theta)+\kappa_1(r\cos\theta,r\sin\theta)
\nu(r\cos\theta,r\sin\theta)$$
defines a\/ $C^\infty$-map\/
$(\M_0,\E_0)\to\R^3$.
\end{lemma}

\begin{proof}
Although the first assertion follows from
\cite[Lemma 2.2]{fh-fronts},
we give a proof in our notation.
Let $E,F,G,L,M,N$ be the coefficients of
the first and second fundamental forms,
regarded as functions of $(r,\theta)$.
	The principal curvatures are the solutions of
\eqref{eq:preq}, namely,
\begin{equation}\label{eq:preqap}
A_2k^2+A_1k+A_0=0
\end{equation}
in $k$, where
$A_2=EG-F^2$, 
$A_1=-(EN-2FM+GL)$, 
$A_0=LN-M^2$.
	We denote the discriminant of \eqref{eq:preqap} by $D=A_1^2-4A_2A_0$.
Then the principal curvatures are given by
\begin{equation}\label{eq:kpm}
k_\pm
=
\dfrac{-A_1\pm\sqrt{D}}{2A_2}
=
\dfrac{-A_1^2+ D}{2A_2(A_1\pm\sqrt{D})}
=
\dfrac{-2A_0}{A_1\pm\sqrt{D}}.
\end{equation}
If $f$ is written as in the form  \eqref{eq:wunormal},
then $D$ is a $C^\infty$-function satisfying
$$
D|_{r=0}=
\dfrac{4 \cos^2\theta f_{331}(0)^2}{\delta^2},
$$
where $\delta=\sqrt{\cos^2\theta f_{331}(0)^2+4 \sin^2\theta}$.
We set
$\tilde{D}(\theta)=
4 f_{331}(0)^2/\delta^2$
and
$$D_1(r,\theta)=D(r,\theta)-\cos^2\theta \tilde{D}(\theta).$$
Thus setting $\pm|\cos\theta|=\cos\theta$, 
$$
-A_1\pm\sqrt{D}=
-A_1+\cos\theta\sqrt{\tilde{D}(\theta)+D_1(r,\theta)/\cos^2\theta}
$$
holds and we see
$$
A_1|_{r=0}=
\dfrac{2 \cos\theta f_{331}(0)}
{\delta},\quad
D_1(0,\theta)=0.
$$
Thus taking a representative 
$(r,\theta)\in\M\cap\{(r,\theta)\,|\,\cos\theta\ne0\}$
with $\cos\theta>0$, we obtain
\begin{equation}\label{eq:k1}
\kappa_1
=
\dfrac{-2A_0}{A_1+\cos\theta\sqrt{\tilde{D}+D_1/\cos^2\theta}}
\end{equation}
	which is the bounded principal curvature.
Since the denominator does not vanish on a sufficiently small neighborhood of $\E_0$,
it follows that $\kappa_1$ is a $C^\infty$-function.

The second assertion follows immediately from the definition of $h_1$ and the smoothness of $\kappa_1$ and $\nu$.
\end{proof}


\subsection{Pedal curve}

 Let $\gamma:(\R,0)\to(\R^2,0)$ be a frontal-germ.
	We denote by $\n$ a unit normal vector field of $\gamma$.
	We fix a point $p\in\R^2$.
	Then the pedal curve $P_{(\gamma,p)}$ of $\gamma$ with respect to $p$
	is defined by
	$$
	P_{(\gamma,p)}=\gamma+
	\dfrac{\e\cdot(p-\gamma)}{\e\cdot\e}\e,
	$$
	where $\e$ is a nowhere-vanishing vector field perpendicular to $\n$.
	
	We prove Theorem \ref{thm:pricipalpedal}.
	
	\begin{proof}[Proof of Theorem \ref{thm:pricipalpedal}]
		Let us assume that $f$ is given in the form \eqref{eq:wunormal}.
		Since $f_u(0,0)=(1,0,0)$ and $f_v(0,0)=(0,0,0)$,
		we have
		$\first(X,X)=a^2$ for $X=a\partial_u+b\partial_v$.
		Thus $a=\pm1$ and $l=\langle(1,0,0)\rangle$.
		
		By
		$f_{uu}(0,0)=2(0,f_{21}(0),f_{31}(0))$,
		$f_{uv}(0,0)=(0,0,f_{331}(0))$,
		and
		$f_{vv}(0,0)=(0,2,0)$,
		we have
		$$
		\second(X,X)=2a^2\trans{\pmt{f_{21}(0)\\f_{31}(0)}}
		+2ab\trans{\pmt{0\\f_{331}(0)}}
		+2b^2\trans{\pmt{1\\0}},
		$$
		where $\trans{(~)}$ is transposition.
		Thus the curvature parabola is parametrized by
		$$
		c(b)=2(f_{21}(0)+b^2,f_{31}(0)+bf_{331}(0)).
		$$
	Since $c'(b)=2(2b,f_{331}(0))\ne0$, the pedal $\trans{P_{(\gamma,0)}}$
	is
	\begin{align}
		&
		2\pmt{f_{21}+b^2\\ f_{31}+bf_{331}}
		-2\dfrac{2b^3+b(2f_{21}+f_{331}^2)+f_{331}f_{31}}{4b^2+f_{331}^2}
		\pmt{2b\\ f_{331}}\nonumber\\
		=&
		\dfrac{2}{4b^2+f_{331}^2}
		\pmt{-f_{331}(-f_{21}f_{331}+2f_{31}b+f_{331}b^2)
			\\    2b     (-f_{21}f_{331}+2f_{31}b+f_{331}b^2)}
		\label{eq:pedcurp}
	\end{align}
	
		Here and throughout the proof, we omit the evaluation values $(0)$ and $(0,0)$ of the functions.
		On the other hand, for the principal curvature surface,
		we have
$E=1$, $F=G=0$,
$L=(-2 \cos\theta f_{21} f_{331}+4 f_{31} \sin\theta)/\delta$, 
$M=2 f_{331} \sin\theta/\delta$, $N=-2 f_{331} \cos\theta/\delta$
at $0$
and
$(\kappa_1)|_{r=0}=
2 (-f_{21} f_{331}\cos^2\theta+2 f_{31} \cos\theta\sin\theta+f_{331} \sin^2\theta)
/
(\cos\theta\delta)$.
Thus, considering $(h_{\kappa_1})|_{r=0}$ 
given in Lemma \ref{lem:k1cinf},
as a map into $l^\perp$, we obtain
\begin{align}
\trans{(h_{\kappa_1})|_{r=0}}&=
2 \dfrac{-f_{21} f_{331}\cos^2\theta+2 f_{31} \cos\theta\sin\theta
+f_{331} \sin^2\theta}
{\cos\theta\delta^2}
\pmt{
-\cos\theta f_{331}\nonumber\\
2 \sin\theta}\\
&=
2 \dfrac{-f_{21} f_{331}+2 f_{31} \tan\theta
+f_{331} \tan^2\theta}
{f_{331}^2+4\tan^2\theta}
\pmt{
-f_{331}\\
2 \tan\theta}.\label{eq:pricur}
\end{align}
Setting $b=\tan\theta$, we see that 
\eqref{eq:pedcurp} coincides with
\eqref{eq:pricur}.
	\end{proof}

Since the focal conic of a Whitney umbrella is the set of  points 
$x=(x_1,x_2,x_3)\in\R^3$
for which
$d_x(u,v)=|f(u,v)-x|^2$ has
a degenerate critical point at $(u,v)=(0,0)$, 
 its defining equations are $x_1=0$ and 
$FC(x_2,x_3)=0$, where 
$FC(x_2,x_3)=4 f_{21}(0) x_2^2  + 4 f_{31}(0) x_2 x_3 - f_{331}(0)^2 
x_3^2-2x_2$.
\begin{proof}[Proof of Theorem \ref{cor:cocalinv}]
Since the focal set is obtained by extending 
in the normal direction by the reciprocal of 
the principal curvature, while the principal curvature surface 
is obtained by extending in the normal direction 
by the principal curvature itself, 
the claim follows immediately from Theorem \ref{thm:pricipalpedal}. 
We include here a proof based on an explicit computation.

By \eqref{eq:pedcurp}, the inversion of the pedal with respect
to the origin of the curvature parabola $\gamma$ is parameterized by
\begin{align*}
i\big(P_{(\gamma,0)}(b)\big)&\\
& \hspace{-18mm}=\left(
-\frac{f_{331}(0)}{4 b f_{31}(0) + 2 b^2 f_{331}(0) - 2 f_{21}(0) f_{331}(0)},
\frac{b}{2 b f_{31}(0) + b^2 f_{331}(0) - f_{21}(0) f_{331}(0)}\right).
\end{align*}
A direct computation shows that  $FC\big(i\big(P_{(\gamma,0)}(b)\big)\big)=0$.
Therefore,
$i\big(P_{(\gamma,0)}(\R)\big)$
coincides with the focal conic, proving the assertion.
\end{proof}

\section{Geometry of the principal curvature surface around Whitney umbrella}

We study the geometry of the principal curvature surface
around a Whitney umbrella.
Let $f:(\R^2,0)\to(\R^3,0)$ be a Whitney umbrella,
and let $h=h_{\kappa_1}$ denote the principal curvature surface associated with
the bounded principal curvature $\kappa_1$.
Since the calculations are quite involved and the resulting formulas are lengthy,
we use the software Mathematica throughout this paper.
For the same reason,
we compute geometric invariants only up to non-zero scalar factors;
that is, we determine only whether each invariant vanishes or not.
Let $x$ and $y$ be functions or vector-valued functions.
If there exists a nowhere-vanishing function $\lambda$ such that
$y=\lambda x$, then we say that
$x$ and $y$ are {\it proportional} and denote
this equivalence relation by $x\simeq y$.

\subsection{Geometric invariants of a Whitney umbrella}
As we mentioned in Section \ref{sec:priwu},
$W_{21}(f)$,
$W_{31}(f)$,
$W_{331}(f)\in\R$ are 
the coefficients 
$f_{21}(0)$,
$f_{31}(0)$,
$f_{331}(0)$ when  $f$ is written in the form
\eqref{eq:wunormal}.
These values
$W_{21}(f)$,
$W_{31}(f)$,
$W_{331}(f)$ are geometric invariants of an $S$-type germ
$f:(\R^2,0)\to(\R^3,0)$.
Formulas for these invariants were originally given in \cite[(12,13,14)]{hhnuy},
and the form  used in this paper is given in 
\cite[(3.11)]{shimada2para}.
Let $(\xi,\eta)$ be a pair of adapted vector fields
(see Section \ref{sec:priwu}). 
Then
\begin{align}
\label{eq:invs}
W_{21}(f)
&\simeq
(f_\xi\cdot f_{\eta\eta})^3(f_\xi\cdot f_{\xi\xi})
-
(f_\xi\cdot f_{\eta\eta})^2
\Big((f_\xi\cdot f_\xi)(f_{\xi\xi}\cdot f_{\eta\eta})+
(f_{\xi}\cdot f_{\xi\eta})^2\Big)\\
&\hspace{5mm}
+
(f_\xi\cdot f_{\eta\eta})(f_\xi\cdot f_\xi)
\Big(2(f_{\xi}\cdot f_{\xi\eta})(f_{\xi\eta}\cdot f_{\eta\eta})
-(f_\xi\cdot f_{\xi\xi})(f_{\eta\eta}\cdot f_{\eta\eta})\Big)
\nonumber\\
&\hspace{5mm}
+
(f_\xi\cdot f_\xi)^2 
\Big((f_{\xi\xi}\cdot f_{\eta\eta})\, (f_{\eta\eta}\cdot f_{\eta\eta})
  - (f_{\xi\eta}\cdot f_{\eta\eta})^2\Big),
\nonumber\\
W_{31}(f)
&\simeq
2\det(f_\xi, f_{\xi\eta}, f_{\eta\eta})
\Big((f_\xi \cdot f_\xi)(f_{\xi\eta} \cdot f_{\eta\eta})
-(f_\xi \cdot f_{\xi\eta})(f_\xi \cdot f_{\eta\eta})\Big)\nonumber\\
&\hspace{5mm}
-(f_\xi \times f_{\eta\eta})\cdot(f_\xi \times f_{\eta\eta})
\det(f_\xi, f_{\xi\xi}, f_{\eta\eta}),\nonumber\\
W_{331}(f)
&\simeq
\det(f_{\xi},f_{\xi\eta},f_{\eta\eta}),\nonumber
\end{align}
where all functions are evaluated at $0$.
We remark that if $f$ satisfies $\rank df_0 = 1$, then
$f$  is a Whitney umbrella at $0$,
if and only if $W_{331}(f)\neq 0$.

\subsection{Singularities of the principal curvature surface}
We set $t=\tan\theta$ and assume $\cos\theta\ne0$.
	We assume that $f$ is written in the normal form \eqref{eq:wunormal}, and 
we set $h=(h_1,h_2,h_3)$.
	Then we have $h_{1\theta}(0,\theta)=0$.
The following lemma holds.
\begin{lemma}\label{eq:singh}
{\rm (1)} The map $h$ has a singular point at $(0,\theta)$
if and only if either
$(S1)$ or $(S2)$ holds, where
\begin{enumerate}
\item[$(S1)$] $h_{1r}\ne0$ and $h_{2\theta}=h_{3\theta}=0$,
\item[$(S2)$]
$
h_{1r}=0$ and $\D=0$, with
$$
\D=\det\pmt{
h_{2r},h_{3r}\\
h_{2\theta},h_{3\theta}}\Bigg|_{r=0}.
$$
\end{enumerate}
{\rm (2)} The condition\/ $(S1)$ is equivalent to\/ 
$f_{21}f_{331}^2+f_{31}^2=0$ and\/ $f_{31}+t f_{331}=0$. In
this case, $h$ is of\/ $S$-type.
If\/ $t\ne0$, then\/ $(S2)$ is equivalent to
\begin{equation}\label{eq:f31f31p}
\begin{array}{rl}
f_{31} &\displaystyle
= \frac{-2 t^2 f_{331} + 2 f_{21} f_{331} \pm \sqrt{4 t^2 + f_{331}^2}}{4 t}\Bigg|_{r=0},\\[4mm]
f_{31}' &\displaystyle
= \frac{-2 t^3 f_{321} - 2 t^4 f_{322} + f_{331} f_{21}' - 2 t^2 f_{331}'}{2 t}\Bigg|_{r=0}.
\end{array}
\end{equation}
In this case, $h$ is of\/ $S$-type if and only if\/
$X\ne0$, where
\begin{align*}
X&=(1+4(t^2+f_{21})^2)f_{331}(4t^2+f_{331}^2)^2\\
&+4\Bigl(
-72t^8f_{322}^2
-16t^7(6f_{321}f_{322}+f_{331}f_{322,v})\\
&\hspace{5mm}
-16t^6(2f_{321}^2+3f_{322}f'_{331}+f_{331}f_{322,u})
-16t^5(f_{331}f'_{321}+2f_{321}f'_{331})\\
&\hspace{10mm}
+4t^4(f_{331}^2-2(f'_{331})^2-2f_{331}f''_{331})
-8t^3f_{331}f''_{31}\\
&\hspace{15mm}
+t^2f_{331}^2(f_{331}^2+4(f_{21}+f''_{21}))
+f_{21}f_{331}^4
\Bigr)\sqrt{4 t^2 + f_{331}^2}\Bigg|_{r=0}.
\end{align*}
If\/ $t=0$, then\/ $(S2)$ is equivalent to\/
$f_{21}=\pm1/2$, $f_{21}'=0$.
In this case, $h$ is of\/ $S$-type if and only if\/
$-1-4 f_{31}^2\pm 4f_{21}''\ne0$.
\end{lemma}
\begin{proof}
	Assertion {\rm (1)} follows immediately from the fact that
$h_{1\theta}(0,\theta)=0$.
We prove {\rm (2)}.
We first consider the case $(S1)$.
	The functions $h_{2\theta}$ and $h_{3\theta}$ are polynomials in $t$.
The resultant of them is
$$
-4 f_{331}^4 \,\bigl(f_{31}^2 + f_{21} f_{331}^2\bigr)\,
\Bigl(16 f_{31}^2 + (4 f_{21} + f_{331}^2)^2\Bigr)^2.
$$
Thus $h_{2\theta}=h_{3\theta}=0$ has a common root if and only if
$(S1$-$1)$: $f_{21}=-f_{31}^2/f_{331}^2$ or $(S1$-$2)$:
$f_{31}=4 f_{21} + f_{331}^2=0$.
We first consider the case $(S1$-$1)$.
We assume $f_{21}=-f_{31}^2/f_{331}^2$.
Under this condition, 
$h_{1r}=f_{331}/\sqrt{f_{31}^2 + f_{331}^2}\ne0$
and
\begin{align*}
h_{\theta}&=
\Bigg(
0,
\dfrac{4(1+t^2)(f_{31}+t f_{331})(4t f_{31}-f_{331}^3)}
     {(4t^2+f_{331}^2)^2},\\
&\hspace{10mm}
\dfrac{4(1+t^2)(f_{31}+t f_{331})(-4t^2 f_{31}+4t^3 f_{331}+f_{31}f_{331}^2+3t f_{331}^3)}{f_{331}(4t^2+f_{331}^2)^2}
\Bigg)\\
&\simeq
(f_{31}+t f_{331})\Big(0,f_{331}(4t f_{31}-f_{331}^3),
(-4t^2 f_{31}+4t^3 f_{331}+f_{31}f_{331}^2+3t f_{331}^3)
\Big)
\end{align*}
hold on $\{r=0\}$. 
Thus $h_{2\theta}=0$ if and only if 
$f_{31}+t f_{331}=0$ or $4t f_{31}-f_{331}^3=0$.
Assume $f_{31}+t f_{331}\ne0$ and $4t f_{31}-f_{331}^3=0$.
If $t=0$, then this is impossible.
Substituting $f_{31}=f_{331}^3/(4t)$ into 
$-4t^2 f_{31}+4t^3 f_{331}+f_{31}f_{331}^2+3t f_{331}^3$,
we obtain 
$f_{331}(4t^2+f_{331}^2)^2/(4t)$, which is non-zero.
Hence $h_{\theta}=0$ if and only if $f_{31}+t f_{331}=0$.
This shows the first assertion concerning $(S1)$.
Substituting $t=-f_{31}/f_{331}$ and $f_{21}=-f_{31}^2/f_{331}^2$
into $h_{\theta\theta}$, we obtain
\begin{equation}\label{eq:hoo}
h_{\theta\theta}=
\left(0,-\dfrac{4(f_{31}^2+f_{331}^2)^2}{4f_{31}^2+f_{331}^4},
-\frac{8f_{31}(f_{31}^2+f_{331}^2)^2}{f_{331}^2(4f_{31}^2+f_{331}^4)}
\right)
\simeq
\left(0,f_{331}^2,2f_{31}\right)
\end{equation}
holds on $\{r=0\}$.
Thus $h$ is of $S$-type.

We next show the case $(S1$-$2)$ does not occur.
Substituting $f_{31}=0$ and $f_{21}=-f_{331}^2/4$ into $h_\theta$,
we obtain
$$h_\theta \simeq (0, 0, f_{331}).$$
So, this case  cannot occur.

We now consider the case $(S2)$.
We see
\begin{align}\label{eq:s2}
h_{1r}=&
  4 f_{331}^{2} t^{4}
+ 16 f_{31} f_{331} t^{3}
+ 4(-1 + 4 f_{31}^{2} - 2 f_{21} f_{331}^{2}) t^{2}
\\
&\hspace{10mm}
- 16 f_{21} f_{31} f_{331} t
- f_{331}^{2}
+ 4 f_{21}^{2} f_{331}^{2},\nonumber\\
\D=&\D_1\D_2,\nonumber\\
\D_1\simeq &
t^2 f_{331} + 2 t f_{31} - f_{21} f_{331},\nonumber\\
\D_2\simeq &2 t^4 f_{322}+ 2 t^3 f_{321}  + 2 t^2 f_{331}'
+ 2 t f_{31}'- f_{331} f_{21}'\nonumber
\end{align}
on $\{r=0\}$.
The resultant of $h_{1r}$ and 
$\D_1$ 
with respect to $f_{21}$
is $4 t^2 f_{331}^2 + f_{331}^4$.
Thus they do not vanish simultaneously.
When $t\ne0$, the equations $h_{1r}=
\D_2=0$ form a quadratic system in the variables $f_{31}$ and $f_{31}'$.
Its solution is given by  \eqref{eq:f31f31p}.
Assume that $t\ne0$ and $h_{1r}=
\D_2=0$. Under the condition \eqref{eq:f31f31p} and $h_{3\theta}=0$,
we see
$h_{2\theta}=(4t^2+f_{331}^2)^{3/2}/(2f_{331})$.
Thus $\rank dh_0\geq1$.
We set $\xi=\partial_\theta$ and 
$\eta=
\sqrt{1+t^2}(1+t^2) f_{331}\partial_r
+t\left(2t f_{321} + 3t^2 f_{322} + f'_{331}\right)\partial_\theta$.
Then $(\xi,\eta)$ is an adapted pair,
A direct calculation shows that
$$
h_\xi\times h_{\eta\eta}=(X,0,0).
$$
	Hence $h$ is of $S$-type if and only if $X\ne0$.

Finally, assume that $t=0$. Then $h_{1r}=
\D_2=0$ is equivalent to $f_{21}=\pm1/2$ and $f_{21}'=0$.
Under these conditions,
$h_{3\theta}\ne0$ and 
$h_\theta\times h_{rr}=(-1-4 f_{31}^2\pm 4f_{21}'',0,0)\ne0$.
Thus the assertion holds.
\end{proof}

The invariants $W_{21}(h)$, $W_{31}(h)$,
and $W_{331}(h)$ can be computed whenever $h$ is of $S$-type at $(0,\theta)$.
Here we only give these invariants
for the case $(S1)$.
The other cases can be obtained in the same way 
by choosing appropriate adapted pair of vector fields.
\begin{proposition}\label{prop:invh}
Let\/ $h=h_{\kappa_1}$ be the principal curvature surface
of a Whitney umbrella\/ $f$ with 
respect to the bounded principal curvature\/ $\kappa_1$.
Let\/ $p=(0,\theta)\in \E_0$ be a singular point of $h$ such that\/ $h$ is of\/ $S$-type,
and assume that $p$ satisfies the condition\/ $(S1)$ in Lemma\/ {\rm \ref{eq:singh}}.
Then the invariants\/ $W_{21}(h)$, $W_{31}(h)$,
and\/ $W_{331}(h)$ at\/ $p$ satisfy
\begin{align}
&W_{21}(h)\\
\simeq\,&
144 f_{31}^{10} f_{322}^{2}
-192 f_{31}^{9} f_{321} f_{322} f_{331}
+32 f_{31}^{8}\left(2 f_{321}^{2}
+3 f_{322} f_{331}'
+2 f_{31} (f_{322})_{v}\right) f_{331}^{2}\nonumber\\
&
-64 f_{31}^{7}\left(f_{321} f_{331}'
+f_{31} (f_{322})_{u}\right) f_{331}^{3}
+16 f_{31}^{6}\left(9 f_{31}^{2} f_{322}^{2}
+4 f_{31} f_{321}'
+f_{331}'^{2}\right) f_{331}^{4}\nonumber\\
&
-8 f_{31}^{6}\left(27 f_{31} f_{321} f_{322}
-18 f_{322} f_{21}'
+4 f_{331}''\right) f_{331}^{5}\nonumber\\
&
+16 f_{31}^{5}\left(
-6 f_{321} f_{21}'
+f_{31}\left(5 f_{321}^{2}
+9 f_{322} f_{331}'\right)
+2 f_{31}''
+2 f_{31}^{2} (f_{322})_{v}\right) f_{331}^{6}\nonumber\\
&
-8 f_{31}^{4}\left(
-6 f_{21}' f_{331}'
+f_{31}\left(9 f_{322} f_{31}'
+13 f_{321} f_{331}'\right)
+4 f_{31}^{2} (f_{322})_{u}\right) f_{331}^{7}\nonumber\\
&
+2 f_{31}^{4}\left(
15 f_{31}^{2} f_{322}^{2}
+24 f_{321} f_{31}'
+16 f_{31} f_{321}'
+16 f_{331}'^{2}
+8 f_{21}''\right) f_{331}^{8}\nonumber\\
&
-8 f_{31}^{3}\left(
6 f_{31}^{2} f_{321} f_{322}
+3 f_{31}' f_{331}'
+f_{31}\left(-6 f_{322} f_{21}'
+2 f_{331}''\right)\right) f_{331}^{9}\nonumber\\
&
+f_{31}^{2}\left(
12 f_{21}'^{2}
+f_{31}^{2}\left(19 f_{321}^{2}
+36 f_{322} f_{331}'\right)
-4 f_{31}\left(9 f_{321} f_{21}'
-4 f_{31}''\right)
+4 f_{31}^{3} (f_{322})_{v}\right) f_{331}^{10}\nonumber\\
&
-4 f_{31}^{2}\left(
-6 f_{21}' f_{331}'
+f_{31}\left(6 f_{322} f_{31}'
+7 f_{321} f_{331}'\right)
+f_{31}^{2} (f_{322})_{u}\right) f_{331}^{11}\nonumber\\
&
+2 f_{31}\left(
-6 f_{21}' f_{31}'
+2 f_{31}^{2} f_{321}'
+f_{31}\left(9 f_{321} f_{31}'
+5 f_{331}'^{2}
+4 f_{21}''\right)\right) f_{331}^{12}\nonumber\\
&
-2 f_{31}\left(
6 f_{31}' f_{331}'
+f_{31} f_{331}''\right) f_{331}^{13}
+\left(3 f_{31}'^{2}
+2 f_{31} f_{31}''\right) f_{331}^{14}
+f_{21}'' f_{331}^{16},\nonumber
\end{align}
\begin{align}\label{eq:w31h}
W_{31}(h)\simeq\,&\omega
\Big(-12 f_{31}^{5} f_{322}
+8 f_{31}^{4} f_{321} f_{331}
-4 f_{31}^{3} f_{331}' f_{331}^{2}
-4 f_{31}^{3} f_{322} f_{331}^{4}\\
&\hspace{20mm}
+f_{31}\left(3 f_{31} f_{321} - 2 f_{21}'\right) f_{331}^{5}
-2 f_{31} f_{331}' f_{331}^{6}
+f_{31}' f_{331}^{7}\Big)\nonumber
\end{align}
and
\begin{equation}
W_{331}(h)\label{eq:w331h}
\simeq\,\omega,
\end{equation}
where\/ $\omega=2 f_{31}^{4} f_{322}
-2 f_{31}^{3} f_{321} f_{331}
+2 f_{31}^{2} f_{331}' f_{331}^{2}
-2 f_{31} f_{31}' f_{331}^{3}
- f_{21}' f_{331}^{5}$.
\end{proposition}
\begin{proof}
Since we assume $(S1)$, the pair $(\partial_r,\partial_\theta)$ is 
adapted.
By a direct calculation, we obtain the assertion.
\end{proof}
As a consequence of Lemma \ref{eq:singh} and 
Proposition \ref{prop:invh}, we obtain
the following corollary.

\begin{corollary}
In the setting of Proposition\/ {\rm \ref{prop:invh}},
that is, under the condition $(S1)$ in Lemma {\rm \ref{eq:singh}},
$W_{331}(h)=0$ implies\/
$W_{31}(h)=0$.
\end{corollary}
This follows  directly from
\eqref{eq:w31h} and \eqref{eq:w331h}.
We also provide an alternative proof.
\begin{proof}
If $p=(0,\theta)$ satisfies the condition $(S1)$
of Lemma \ref{eq:singh}, then by \eqref{eq:hoo}
it holds that $h_{\theta\theta}|_{r=0}\simeq(0,f_{331}^2,2f_{31})$.
This implies that 
$2 f_{31}h_{2\theta\theta}-f_{331}^2h_{3\theta\theta}=0$
on $\{r=0\}$.
Moreover, by a direct calculation,
we see $2 f_{31}h_{2r}-f_{331}^2h_{3r}=0$ on $\{r=0\}$,
under the condition $(S1)$.
Furthermore, 
by a calculation, we see
\begin{equation}
\label{eq:hrrhrohoo}
\begin{array}{rcl}
2 f_{31}h_{2rr}-f_{331}^2h_{3rr}
&=&
X_1\omega
,\\[2mm]
2 f_{31}h_{2r\theta}-f_{331}^2h_{3r\theta}
&=&
X_2\omega
\end{array}
\end{equation}
on $\{r=0\}$, where,
\begin{align*}
X_1
=&
-\frac{
24 f_{31}
\left(3 f_{31}^2 f_{322} - 2 f_{31} f_{321} f_{331} + f_{331}^2 f_{331}'\right)
}{
f_{331}^2 \left(f_{31}^2 + f_{331}^2\right)\left(4 f_{31}^2 + f_{331}^4\right)
},\\
X_2
=&
-\frac{
12 \sqrt{f_{31}^2 + f_{331}^2}
}{
4 f_{31}^2 f_{331} + f_{331}^5
}.
\end{align*}
We set
$$C_1=(h_r,h_{rr},h_{\theta\theta})
\quad\text{and}\quad
C_2=(h_r,h_{r\theta},h_{\theta\theta}).$$
Then \eqref{eq:hrrhrohoo} implies that
by considering the row operation that subtracts 
$f_{331}^2$ times the third row from 
$2 f_{31}$ times the second row of 
$C_1$ and $C_2$,
we see that $C_1$ and $C_2$ are modified into
$$
\pmt{
h_{1r}&h_{1rr}   &h_{1\theta\theta}\\
h_{2r}&h_{2rr}   &h_{2\theta\theta}\\
0    &X_1\omega&0}
\quad\text{and}\quad
\pmt{
h_{1r}&h_{1r\theta}   &h_{1\theta\theta}\\
h_{2r}&h_{2r\theta}   &h_{2\theta\theta}\\
0    &X_2\omega     &0}
$$
respectively.
Thus $\omega$ 
divides both
$\det C_1$
and
$\det C_2$. On the other hand,
$W_{31}(h)$ is a linear combination of $\det C_1$ and $\det C_2$,
and $W_{331}(h)$ equals $\det C_2$.
Thus $\omega$ 
divides both $W_{31}(h)$ and $W_{331}(h)$.
This implies the assertion. 
\end{proof}
It should be remarked that as one can see from \eqref{eq:wunormal},
for a general $S$-type germ $g$, 
the conditions $W_{31}(g)=0$ and $W_{331}(g)=0$ 
are distinct.
In fact, if $g=(u,v^2,au^2+buv)$, then
$W_{31}(g)=a$ and $W_{331}(g)=b$.


\section{Example}\label{sec:ex}
Set $ f=(u, u^3 + v^2, u^2 + uv + v^3). $
Then  $\rank df_0=1$. Moreover,
$ f_{vv}(0,0)=(0,2,0)$ and $f_{uv}(0,0)=(0,0,1).$
Hence
$$ \det\bigl(f_u(0,0),f_{uv}(0,0),f_{vv}(0,0)\bigr)=-2\neq0.$$
Therefore $f$ is a Whitney umbrella. Then we have
$\nu_2(r,\theta)=f_u\times f_v(r\cos\theta,r\sin\theta)/r$,
where
$$
\nu_2=
(3 r^2 \cos^3\theta-4 r \cos\theta\sin\theta-2 r \sin^2\theta+9 r^3 \cos^2\theta\sin^2\theta,-\cos\theta-3 r \sin^2\theta,2 \sin\theta).
$$
We set $\delta=|\nu_2|$ and 
$
\nu=\nu_2/\delta$. Using the notation of Section \ref{sec:cpara}
and the proof of Lemma \ref{lem:k1cinf}, we obtain
\begin{align*}
E&=
1 + 9 r^4 \cos^4\theta + (2 r \cos\theta + r \sin\theta)^2,\\
F&=r^2\left(6 r \cos^2\theta \sin\theta 
+ (2 \cos\theta + \sin\theta)(\cos\theta + 3 r \sin^2\theta)\right),\\
G&=r^2\left(4 \sin^2\theta + (\cos\theta + 3 r \sin^2\theta)^2\right)
\end{align*}
and
$$
\secondmat=
\pmt{
  \dfrac{-6 r - 9 r^2 \cos\theta - 6 r \cos 2\theta + 9 r^2 \cos 3\theta + 8 \sin\theta}{2\delta} 
  & \dfrac{2 \sin \theta}{\delta} \\[8pt]
  \dfrac{2 \sin \theta}{\delta} 
  & -\dfrac{2 (\cos \theta-3r\sin^2\theta)}{\delta}
}.
$$
As we saw in Section \ref{sec:priwu}, the bounded principal curvature is
$$
\kappa_1=\dfrac{-2(LN-M^2)}{-(EN-2FM+GL)+\cos\theta
\sqrt{\dfrac{D}{\cos^2\theta}}},
$$
with $D/\cos^2\theta|_{r=0}=4/(\cos^2\theta+4\sin^2\theta)$
and 
$-(EN-2FM+GL)|_{r=0}=2\cos\theta/(\cos^2\theta+4 \sin^2\theta)^{1/2}$.
Thus  $\kappa_1$ is a 
$C^\infty$ function, and so is $h=f+\kappa_1\nu$.
See Figure \ref{fig:figprisurf} for the image of $h$.
The thick curve represents the image of the exceptional set $h(\E_0)$.
\begin{figure}[htbp]
\centering
\includegraphics[width=0.4\linewidth]{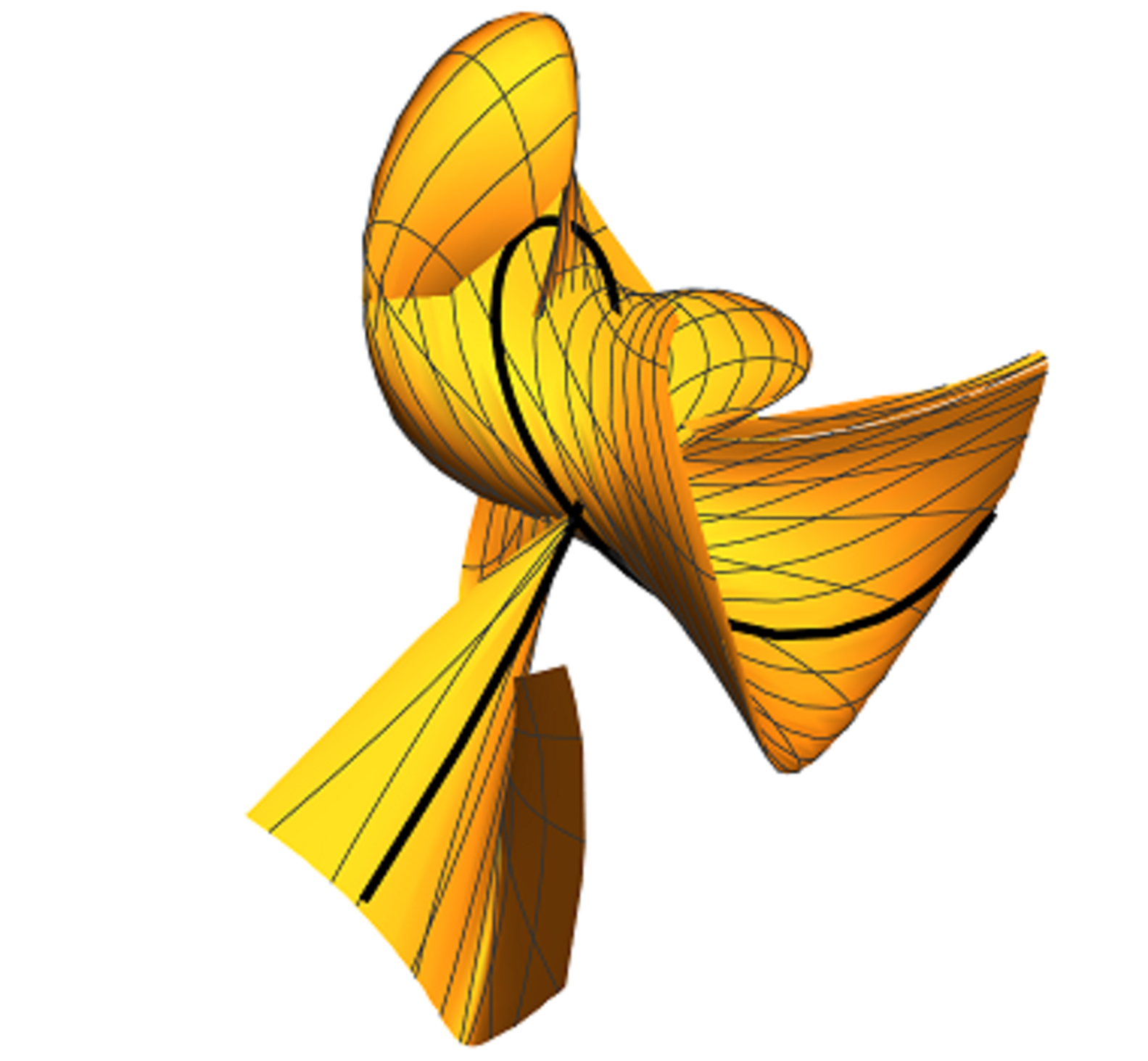}
\caption{Principal curvature surface $h$ of $f$}
\label{fig:figprisurf}
\end{figure}
Figure \ref{fig:figcpandpedcp} left shows
the curvature parabola of $f$ (dashed line)
and its pedal curve with respect to the origin (solid line). 
As shown in Theorem \ref{thm:pricipalpedal}, the latter coincides with 
$h(\E_0)$.
\begin{figure}[htbp]
\centering
\includegraphics[width=0.3\linewidth]{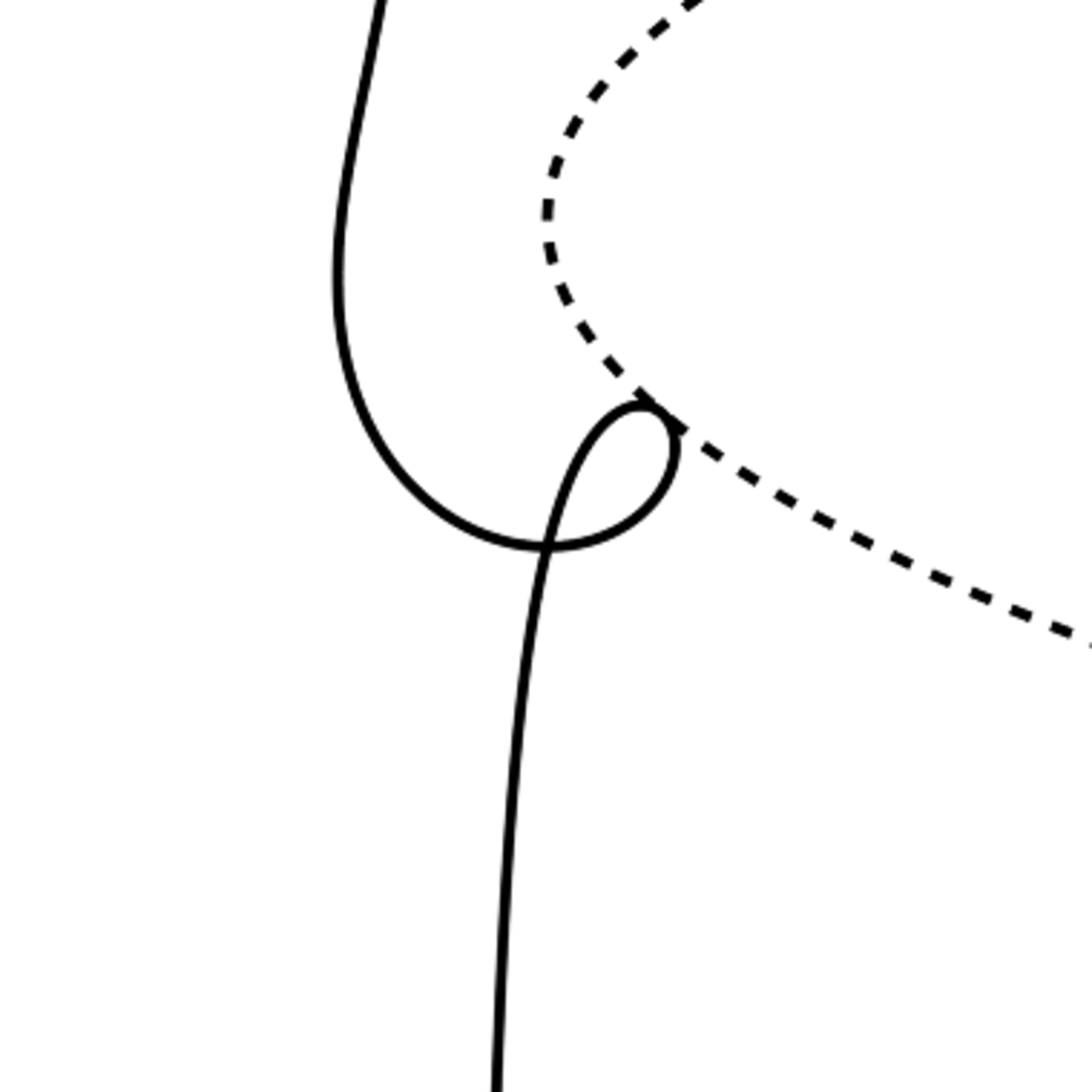}
\hspace{5mm}
\includegraphics[width=0.3\linewidth]{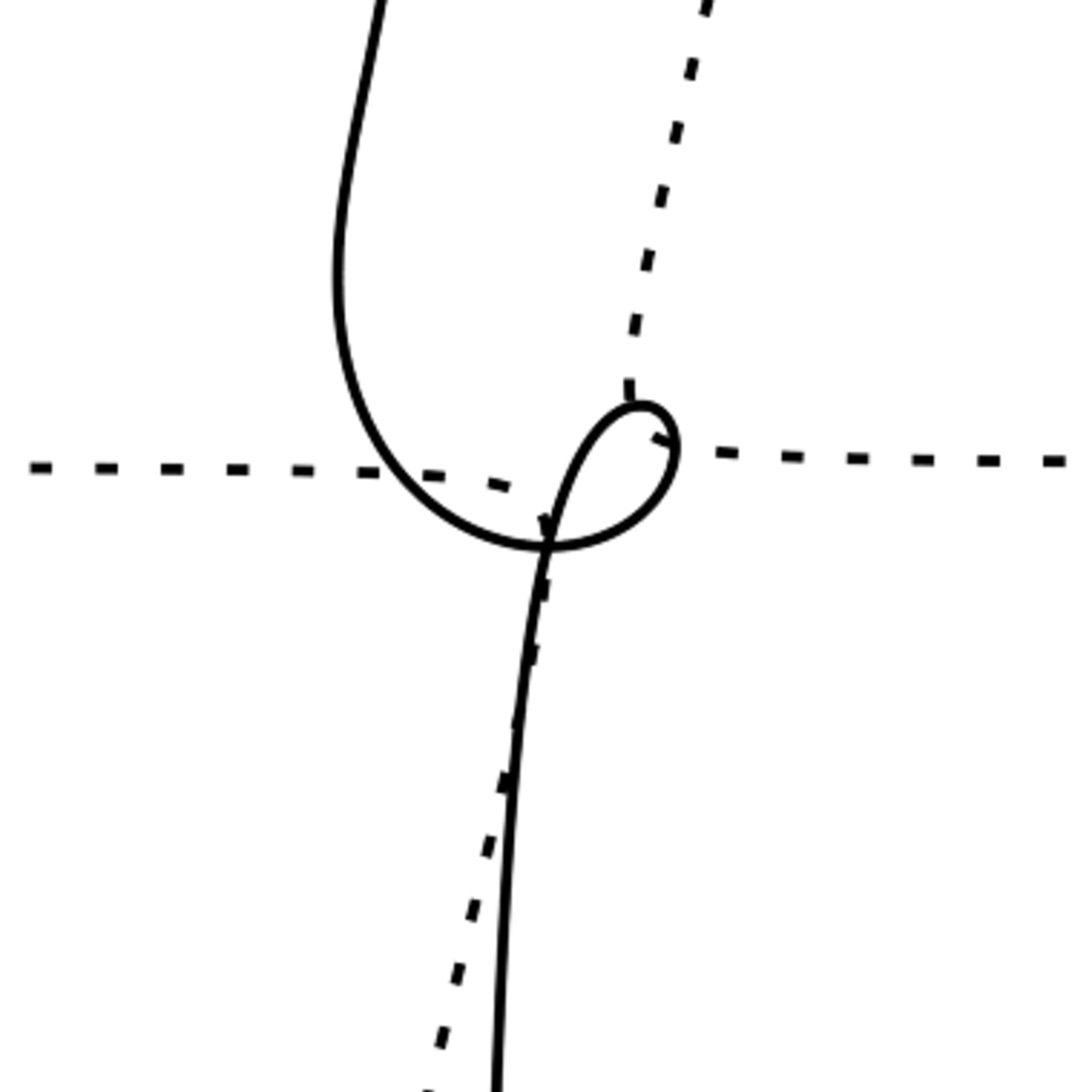}
\caption{Curvature parabola, 
pedal of curvature parabola and focal conic of $f$}
\label{fig:figcpandpedcp}
\end{figure}
Figure \ref{fig:figcpandpedcp} right shows the pedal curve  of 
the curvature parabola with respect to the origin (solid line)
and the focal conic of $f$ (dotted line). 
As proved in Theorem \ref{cor:cocalinv}, 
the inversion of the former coincides with the latter.
This surface $h$ satisfies $S(h)\cap \E_0=\emptyset$.
In particular, the curves $\theta\mapsto h(0,\theta)$ and
$r\mapsto h(r,\theta)$ for any fixed $\theta$ have no singular
points at $r=0$.

Setting $t=\tan\theta$, we see that 
the Gaussian and mean curvatures $K$ and $H$ of $h$, respectively, 
satisfy the following:
\begin{align*}
K|_{r=0}\simeq\,&14+22 t+255 t^2-436 t^3+13777 t^4+63678 t^5+107855 t^6
+212040 t^7\\
&\hspace{5mm}+261408 t^8-261360 t^9
-1101292 t^{10}-1336608 t^{11}+77888 t^{12}\\
&\hspace{5mm}+3536320 t^{13}+6374880 t^{14}+6056192 t^{15}+4287872 t^{16}+2984832 t^{17}\\
&\hspace{5mm}+1758080 t^{18}+625152 t^{19}+102912 t^{20}+4096 t^{21},\\
H|_{r=0}\simeq\,&-12+34 t-1183 t^2-3904 t^3+8244 t^4+16088 t^5+43956 t^6+184544 t^7\\
&\hspace{5mm}+340872 t^8
+405248 t^9+362816 t^{10}+214080 t^{11}+135552 t^{12}\\
&\hspace{5mm}+151552 t^{13}+131200 t^{14}+99584 t^{15}
+61184 t^{16}+17408 t^{17}+1024 t^{18}.
\end{align*}
Moreover, 
the geodesic curvature $\kappa_g$ and the normal curvature
 $\kappa_n$ of the curve $h|_{\E_0}$
regarded as a curve on the surface $h$ 
satisfy the following:
\begin{align*}
\kappa_g&\simeq t(2+t)(-4-6 t-3 t^2+8 t^3)(-1+2 t^4),\\
\kappa_n&\simeq (-4-6 t-3 t^2+8 t^3)(-1+12 t^2+16 t^3+4 t^4).
\end{align*}
In contrast, for a fixed $\theta\in \E_0$,
consider the radial curve $r\mapsto h(r,\theta)$ as a curve on $h$.
Its geodesic and normal curvatures at $r=0$ 
are functions of $\theta$.
We denote them by $r_g(\theta)$ and $r_n(\theta)$, respectively.
Then we have
\begin{align*}
r_g&\simeq t\big(-138-191 t+451 t^2-148 t^3+4358 t^4+11837 t^5+28023 t^6+76684 t^7
\\
&\hspace{10mm}+135848 t^8+211032 t^9+285680 t^{10}+335168 t^{11}+365016 t^{12}+279984 t^{13}\\
&\hspace{10mm}
+142320 t^{14}+61248 t^{15}+15440 t^{16}+6272 t^{17}+2816 t^{18}\big),\\
r_n&\simeq -1-2 t+411 t^2+332 t^3+1881 t^4+4264 t^5+7144 t^6+16052 t^7
\\
&\hspace{10mm}+20458 t^8+15264 t^9+8456 t^{10}+3600 t^{11}+1400 t^{12}+448 t^{13}+128 t^{14}.
\end{align*}

Since proportionality does not affect the roots of a polynomial,
the discriminant of a polynomial vanishes if and only if
the discriminant of any proportional polynomial vanishes.
A direct computation shows that the discriminants of 
the above polynomials are all non-zero. 
Consequently, the discriminants of
$K|_{r=0}$, $H|_{r=0}$, $\kappa_g$, $\kappa_n$,
$r_g$ and $r_n$ regarded as polynomials in $t$, are non-zero,
and this implies that all the roots of these polynomials
are simple.

\section{Geometry of the exceptional set as a curve}

Let $f:(\R^2,0)\to(\R^3,0)$ be a Whitney umbrella.
We assume that $f$ is written in the normal form
\eqref{eq:wunormal} and consider the blow-up \eqref{eq:bl}.
Let $h=h_{\kappa_1}$ denote the principal curvature surface associated with
the bounded principal curvature $\kappa_1$.
Since $\E_0$ is a curve on $\M_0$ and $\nu(\E_0)$ is also a curve,
the image $h(\E_0)$ is a curve as well.
As we have shown, it coincides with the pedal curve of the curvature parabola.
The geometric behavior of $h(\E_0)$ as a curve on
the surface $h(\M_0)$ reflects
the corresponding geometry of the original Whitney umbrella.

In this section, we study the geometric properties of
$h(\E_0)$ related to its geodesic curvature and normal curvature,
as well as the Gaussian and mean curvatures of $h$ along $\E_0$.

\subsection{Space of Whitney umbrella up to four-jet}
Since the Whitney umbrella is two-determined,
and all the conditions and invariants 
we consider depend only on terms of degree at most four,
it suffices to restrict our attention to four-jets of  Whitney umbrellas.
We set 
\begin{align*}
	X&=J^2(1,1)\times J^2(1,1)\times J^1(1,1)\times J^1(2,1)\times J^2(1,1)_+\\
	&=
	\{j=(j^2f_{21}(0),j^2f_{31}(0),j^1f_{321}(0),j^1f_{322}(0,0),j^2f_{331}(0))\,|\,f_{331}(0)>0
	\},
\end{align*}
where $J^r(m,n)$ is the $r$-jet space at the origin of $\R^m$.
Since the normal form \eqref{eq:wunormal} is unique,
there is a one-to-one correspondence between $X$ and
$\{j^4 f(0) \in J^2(2,3) | f \text{ is a Whitney}$ $\text{umbrella} \}/\!\sim$, 
where $\sim$ denotes the equivalence relation given by
the natural action of diffeomorphisms of $(\R^2,0)$ and of $SO(3)$.
In fact, for a four-jet $j^4 f(0)$ of a Whitney umbrella, 
we express $f$ by the form \eqref{eq:wunormal}.
Then the functions 
$f_{21},f_{31},f_{321},f_{322}$ and $f_{331}$ appear.
Associating these functions with each element of $X$,
we may identify $X$ with the space of Whitney umbrella up to four-jet.
We regard $X$ as a subset of 
a Euclidean space of suitable dimension in the usual way.

Here we see the condition for singularities of $h$ at
the exceptional set.
Let $e(\theta)=(0,\theta)$  be a parameterization of the
exceptional set, and $\hat{e}:\E_0\to\R^3$ defined by
$\hat{e}(\theta)=h\circ e(\theta)=h(0,\theta)$ 
be a parameterization of the exceptional set on 
the principal curvature surface.
The condition $\hat{e}'\cdot \hat{e}'=0$ is equivalent to
$h_\theta\cdot h_\theta|_{r=0}=0$.
By Lemma \ref{eq:singh} $(S1)$,
if $f_{21}f_{331}^2+f_{31}^2\ne0$,
then $\hat{e}'\cdot \hat{e}'\ne0$ for any $\theta$.
We set 
\begin{equation}\label{eq:z1}
	Z_1=f_{21}f_{331}^2+f_{31}^2,
\end{equation}
where we omit the evaluating value $(0)$ or $(0,0)$ for
functions.
On the other hand, 
for a fixed $\theta$, a curve $\rho:r\mapsto (r,\theta)$ is called
a {\it radial curve with respect to $\theta$},
and we set $\hat\rho=h\circ\rho$.
Although the construction depends on the chosen blow-up, 
the normal form \eqref{eq:wunormal} is unique, and the blow-up
is natural, so the resulting curve has a certain geometric meaning.
We next determine  the condition under which $\hat\rho$ is singular.
Under the notation of Lemma \ref{eq:singh},
the resultants of the pairs 
$(h_{1r},h_{2r})$ and $(h_{1r},h_{3r})$,
viewed as polynomials in $t$,
are polynomials in
$f_{21}$, $f_{31}$, $f_{331}$, 
$f_{21}'$, $f_{31}'$, $f_{331}'$, $f_{321}$, $f_{322}$.
We denote them by $Z_2$ and $Z_3$ respectively.
We next see the condition $\E_0\cap S(h)=\emptyset$. 
Under the notation of Lemma \ref{eq:singh}, because of \eqref{eq:s2},
the condition  $(S2)$ is equivalent to
$(h_{1r},\D_1)=(0,0)$ or $(h_{1r},\D_2)=(0,0)$.
The resultant of the pair $(h_{1r},\D_1)$,
considered as polynomials in $t$,
is proportional to
\begin{equation}\label{eq:z2}
	Z_4=(4 f_{21} + f_{331}^{2})^2+16f_{31}^{2}.
\end{equation}
In the same manner, the resultant of the pair
$(h_{1r},\D_2)$ is
a polynomial in
$f_{21}$, $f_{31}$, $f_{331}$, 
$f_{21}'$, $f_{31}'$, $f_{331}'$, $f_{321}$, $f_{322}$.
We denote this by $Z_5$.
For a polynomial map $Z$ on $X$,
we denote by  $\ZZ(Z)$ the zero set of $Z$.
By the construction, we have the following lemma.
\begin{lemma}\label{lem:singe}
	Under the above notation, if\/ $f\in X\setminus\ZZ(Z_1)$,
	then the corresponding principal curvature surface\/ $h$ satisfies
	that the curve\/ $\hat{e}=h(0,\theta)$ is nonsingular for every\/ $\theta\in \E_0$.
	If\/ $f\in X\setminus(\ZZ(Z_2)\cap \ZZ(Z_3))$,
	then the corresponding principal curvature surface\/ $h$ satisfies
	that the radial curve\/ $\hat{\rho}:r\mapsto h(r,\theta)$ 
	is nonsingular at $r=0$  for every\/ $\theta\in \E_0$.
	If\/ $f\in X\setminus(\ZZ(Z_1)\cup\ZZ(Z_4)\cup\ZZ(Z_5))$,
	then the corresponding principal curvature surface\/ $h$ satisfies
	that\/ $h$ is nonsingular on\/ $\E_0$.
\end{lemma}

\subsection{Geodesic and normal curvatures}
Let $h=h_{\kappa_1}:(\M_0,\E_0)\to\R^3$ be the principal curvature
surface of a Whitney umbrella $f$ with 
respect to the bounded principal curvature $\kappa_1$.
We consider the geodesic curvature $\kappa_g$
and the normal curvature $\kappa_n$
of the curve $h|_{\E_0}$ on the surface $h$,
defined on
$\E_0\setminus S(h|_{\E_0})$.
In this section, we show the following theorem.

\begin{theorem}\label{thm:numkgkn}
	Let\/ $h=h_{\kappa_1}:(\M_0,\E_0)\to\R^3$ be the principal curvature
	surface of a Whitney umbrella\/ $f$ with 
	respect to the bounded principal curvature\/ $\kappa_1$.
	Let\/ $X$ be the space of Whitney umbrellas up to four-jets
	identified with a subset of a Euclidean space.
	
	{\rm (1)} There exists an algebraic subset\/ $A_g\subsetneq X$ such that
	for any\/ $f\in X\setminus A_g$, the curve\/
	$h|_{\E_0}$ has no singular point, and
	the number of zeros of the geodesic curvature\/ $\kappa_g$
	of\/ $h|_{\E_0}$ as a curve on\/ $h$
	is\/ $2n - 1$ for\/ $n = 1, \ldots, 5$.
	
	{\rm (2)} There exists an algebraic subset\/ $A_n\subsetneq X$ such that
	for any\/ $f\in X\setminus A_n$, the curve\/
	$h|_{\E_0}$ has no singular point, and
	the number of zeros of the geodesic curvature\/ $\kappa_g$
	of\/ $h|_{\E_0}$ as a curve on\/ $h$
	is\/ $2n-1$ for\/ $n = 1, \ldots, 4$.
\end{theorem}
\begin{proof}
	Assume that $f$ is written in the normal form
	\eqref{eq:wunormal}.
	Then by using \eqref{eq:k1}, we obtain a concrete form of
	$h=f+\kappa_1\nu$.
	We calculate the geodesic and normal curvatures of 
	the curve $\hat{e}:\theta\mapsto h(0,\theta)$.
	Since 
	$\hat{e}'$ is proportional to
	$$\big(
	0,\;
	4t^2 f_{31}-t(4f_{21}f_{331}+f_{331}^3)-f_{31}f_{331}^2,\;
	4t^4+t^2(4f_{21}+3f_{331}^2)+4t f_{31}f_{331}-f_{21}f_{331}^2
	\big)
	$$
	and $\hat{e}''$ is proportional to
	\begin{align*}
		&\Bigg(
		0,\;
		-4 f_{21} f_{331}^3 - f_{331}^5
		+ (24 f_{31} f_{331}^2 - 2 f_{31} f_{331}^4) t\\
		&\hspace{10mm}
		+ (48 f_{21} f_{331} + 12 f_{331}^3 - 12 f_{21} f_{331}^3 - 3 f_{331}^5) t^2\\
		&\hspace{10mm}
		+ (-32 f_{31} + 24 f_{31} f_{331}^2) t^3
		+ (16 f_{21} f_{331} + 4 f_{331}^3) t^4
		,\\
		&\hspace{10mm}
		2\Big(2 f_{31} f_{331}^3
		+ \big(3 f_{331}^4 - f_{21} f_{331}^2 (-12 + f_{331}^2)\big) t
		+ 6 f_{31} f_{331} (-4 + f_{331}^2) t^2\\
		&\hspace{10mm}
		+ 2 \big(f_{331}^2 (-2 + 3 f_{331}^2) + f_{21} (-8 + 6 f_{331}^2)\big) t^3
		- 8 f_{31} f_{331} t^4
		+ 12 f_{331}^2 t^5
		+ 16 t^7\Big)
		\Bigg).
	\end{align*}
	On the other hand, $\nu|_{r=0}$ is proportional to
	\begin{align*}
		&\Bigg(
		-6(- f_{21} f_{331} + 2 f_{31} t + f_{331} t^2)(4t^2+f_{331}^2)\\
		&\hspace{10mm}
		(- f_{331} f_{21}'
		+ 2 f_{31}' t
		+ 2 f_{331}' t^2
		+ 2 f_{321} t^3
		+ 2 f_{322} t^4),\\
		&
		\bigl(- f_{331}^2 + 4 f_{21}^2 f_{331}^2
		- 16 f_{21} f_{31} f_{331} t
		+ (-4 + 16 f_{31}^2 - 8 f_{21} f_{331}^2) t^2
		+ 16 f_{31} f_{331} t^3
		+ 4 f_{331}^2 t^4\bigr)\\
		&\hspace{10mm}
		\big(- f_{21} f_{331}^2
		+ 4 f_{31} f_{331} t
		+ (4 f_{21} + 3 f_{331}^2) t^2
		+ 4 t^4\big)
		,\\
		&\bigl(- f_{331}^2 + 4 f_{21}^2 f_{331}^2
		- 16 f_{21} f_{31} f_{331} t
		+ (-4 + 16 f_{31}^2 - 8 f_{21} f_{331}^2) t^2
		+ 16 f_{31} f_{331} t^3
		+ 4 f_{331}^2 t^4\bigr)\\
		&\hspace{10mm}
		\big(- f_{31} f_{331}^2 + (-4 f_{21} f_{331} - f_{331}^3) t + 4 f_{31} t^2\big)
		\Bigg).
	\end{align*}
	
	A direct computation shows that the geodesic curvature
	$\kappa_g$, which is proportional to
	$\det(\hat{e}',\hat{e}'',$ $\nu|_{r=0})$,
	is proportional to
	\begin{align*}
		&\Big(- f_{331}(4 f_{21}^2 + 4 f_{31}^2 + f_{21} f_{331}^2)
		- 6(f_{31} f_{331}^2)t
		- 3(4 f_{21} f_{331} + f_{331}^3)t^2
		+ 8 f_{31} t^3\Big)\\
		&\hspace{10mm}
		(- f_{21} f_{331} + 2 f_{31} t + f_{331} t^2)
		(- f_{331} f_{21}' + 2 f_{31}' t + 2 f_{331}' t^2 + 2 f_{321} t^3 + 2 f_{322} t^4),
	\end{align*}
	which is a polynomial in $t=\tan\theta$
	of degree $9$.
	Similarly, the normal curvature $\kappa_n$,
	which is proportional to
	$\hat{e}''\cdot\nu|_{r=0}$,
	is proportional to
	\begin{align*}
		&\Big((-1 + 4 f_{21}^2) f_{331}^2
		- 16f_{21} f_{31} f_{331} t
		+ 4(-1 + 4 f_{31}^2 - 2 f_{21} f_{331}^2) t^2
		+ 16 f_{31} f_{331} t^3
		+ 4 f_{331}^2 t^4\Big)\\
		&\hspace{10mm}
		(- f_{331}(4 f_{21}^2 + 4 f_{31}^2 + f_{21} f_{331}^2)
		- 6f_{31} f_{331}^2t
		- 3(4 f_{21} f_{331} + f_{331}^3)t^2
		+ 8 f_{31} t^3)
	\end{align*}
	which is a polynomial in $t=\tan\theta$
	of degree $7$.
	On the other hand,
	the leading coefficient of $\kappa_g$
	is proportional to
	$f_{31}f_{322}f_{331}$,
	and that of $\kappa_n$
	is proportional to
	$f_{31}f_{331}^2$.
	
	We denote by $Z_g$ and $Z_n$
	the discriminants of
	$\kappa_g$ and $\kappa_n$,
	respectively, regarded as polynomials in $t$.
	We may regard $Z_g$ and $Z_n$
	as polynomial maps on $X$.
	By the example in Section \ref{sec:ex},
	neither of them is the zero polynomial.
	
	We set
	$$
	A_g=\ZZ(Z_1)\cup \ZZ(Z_g)\cup \ZZ(f_{31}f_{322}),
	\qquad
	A_n=\ZZ(Z_1)\cup \ZZ(Z_n).
	$$
	Then, by Lemma \ref{lem:singe},
	every $f\in X\setminus A_g$
	satisfies the assertion of {\rm (1)},
	and every $f\in X\setminus A_n$
	satisfies the assertion of {\rm (2)}.
	This proves the theorem.
\end{proof}

\subsection{Geodesic and normal curvatures of radial curves at $\E_0$}

Let $h=h_{\kappa_1}:(\M_0,\E_0)\to\R^3$ be the principal curvature
surface of a Whitney umbrella $f$ with
respect to the bounded principal curvature $\kappa_1$.
We consider the geodesic and normal curvatures
of a radial curve regarded as a curve on $h$,
evaluated at $\E_0$.
These curvatures, evaluated at $\E_0$, are functions of
$\theta\in\E_0$,
which we denote by $r_g(\theta)$ and $r_n(\theta)$,
respectively.

In this section, we prove the following theorem.

\begin{theorem}\label{thm:numkgknrad}
	In the same setting as Theorem {\rm \ref{thm:numkgkn}},
	the following hold.
	
	{\rm (1)} There exists an algebraic subset\/ $A_{rg}\subsetneq X$ such that
	for any\/ $f\in X\setminus A_{rg}$,
	the radial curve is nonsingular at\/  $\E_0$ for every\/ $\theta$,
	and the number of zeros of\/ $r_g(\theta)$ for\/ $\theta\in\E_0$
	is\/ $2n - 1$ for\/ $n = 1, \ldots, 10$.
	
	{\rm (2)} There exists an algebraic subset\/ $A_{rn}\subsetneq X$ such that
	for any\/ $f\in X\setminus A_{rn}$,
	the radial curve is nonsingular at\/ $\E_0$ for every\/ $\theta$,
	and the number of zeros of\/ $r_n(\theta)$ for\/ $\theta\in\E_0$
	is\/ $2n$ for\/ $n = 0,1, \ldots, 7$.
\end{theorem}

\begin{proof}
	Assume that $f$ is written in the normal form
	\eqref{eq:wunormal}.
	Then, using \eqref{eq:k1}, we obtain an explicit expression for
	$h=f+\kappa_1\nu$.
	
	For a fixed $\theta$, we compute the geodesic and normal curvatures of
	the curve $\rho:r\mapsto h(r,\theta)$ at $r=0$.
	We obtain that
	$r_g$, which is proportional to
	$\det(h_r,h_{rr},\nu)|_{r=0}$,
	is proportional to a polynomial in $t=\tan\theta$
	of degree $19$, and that
	$r_n$, which is proportional to
	$h_{rr}\cdot\nu|_{r=0}$,
	is proportional to a polynomial in $t$
	of degree $14$.
	
	On the other hand,
	the leading coefficient of $r_g$ is proportional to
	$f_{331}^4(9 f_{322}^2+2 f_{331}^2)$,
	and that of $r_n$ is proportional to
	$f_{331}^5$.
	
	We denote by $Z_{rg}$ and $Z_{rn}$
	the discriminants of
	$r_g$ and $r_n$,
	respectively, regarded as polynomials in $t$.
	We may regard $Z_{rg}$ and $Z_{rn}$
	as polynomial maps on $X$.
	By the example in Section \ref{sec:ex},
	neither of them is the zero polynomial.
	
	We set
	$$
	A_{rg}=\left(\bigcap_{i=2,3}\ZZ(Z_i)\right)\cup \ZZ(Z_{rg}),
	\qquad
	A_{rn}=\left(\bigcap_{i=2,3}\ZZ(Z_i)\right)\cup \ZZ(Z_{rn}).
	$$
	Then, by Lemma \ref{lem:singe},
	every $f\in X\setminus A_{rg}$
	satisfies the assertion of {\rm (1)}, and
	every $f\in X\setminus A_{rn}$
	satisfies the assertion of {\rm (2)}.
	This proves the theorem.
\end{proof}

\subsection{Gaussian and mean curvatures}

In this section, we prove the following theorem.
Let
$K=(LN-M^2)/(EG-F^2)$ and $H=(EN-2FM+GL)/(2(EG-F^2))$
be the Gaussian and mean curvatures of $h$ defined on
$\M_0\setminus S(h)$, where
$E,F,G,L,M,N$ are the coefficients of the first and second fundamental forms;
see Section \ref{sec:priwu}.

\begin{theorem}\label{thm:numkh}
	In the same setting as Theorem\/ {\rm \ref{thm:numkgkn}},
	the following hold.
	
	{\rm (1)} There exists an algebraic subset\/ $A_k\subsetneq X$ such that,
	for any\/ $f\in X\setminus A_k$, the principal curvature surface\/
	$h$ has no singular point on\/ $\E_0$, and
	the number of zeros of the Gaussian curvature\/ $K$ on\/ $\E_0$ is\/
	$2n-1$ for\/ $n=1,\ldots,11$.
	
	{\rm (2)} There exists an algebraic subset\/ $A_h\subsetneq X$ such that,
	for any\/ $f\in X\setminus A_h$, the principal curvature surface\/
	$h$ has no singular point on\/ $\E_0$, and
	the number of zeros of the mean curvature\/ $H$ on\/ $\E_0$
	is\/ $2n$ for\/ $n=0,1,\ldots,9$.
\end{theorem}

\begin{proof}
	We see that
	$\tilde{K}=LN-M^2$ is proportional to a polynomial in
	$t=\tan\theta$ of degree $21$, and that
	$\tilde{H}=EN-2FM+GL$ is proportional to a polynomial in
	$t$ of degree $18$.
	We do not give their explicit expressions
	because of their length.
	
	The leading coefficients of $\tilde{K}$ and $\tilde{H}$
	are proportional to
	$f_{31}f_{331}^7$
	and
	$f_{331}^5$,
	respectively.
	
	We denote by $Z_k$ and $Z_h$
	the discriminants of
	$\tilde{K}$ and $\tilde{H}$,
	respectively, regarded as polynomials in $t$.
	We may regard $Z_k$ and $Z_h$
	as polynomial maps on $X$.
	By the example in Section \ref{sec:ex},
	neither of them is the zero polynomial.
	
	We remark that if
	$\E_0\cap S(h)=\emptyset$,
	then
	$EG-F^2$
	never vanishes on $\E_0$.
	We set
	$$
	A_k=\ZZ(Z_4)\cup \ZZ(Z_5)\cup \ZZ(Z_k)\cup \ZZ(f_{31}),
	\qquad
	A_h=\ZZ(Z_4)\cup \ZZ(Z_5)\cup \ZZ(Z_h),
	$$
	where $f_{31}$ is regarded as a polynomial on $X$.
	
	Then, by Lemma \ref{lem:singe},
	every $f\in X\setminus A_k$
	satisfies the assertion of {\rm (1)},
	and every $f\in X\setminus A_h$
	satisfies the assertion of {\rm (2)}.
	This proves the theorem.
\end{proof}

We note that, in the example given in Section \ref{sec:ex},
Theorems \ref{thm:numkgkn},
\ref{thm:numkgknrad},
and \ref{thm:numkh} all apply,
and the degrees of the corresponding functions agree with those predicted by the theorems.

\appendix
\section{Singularities of principal curvature surfaces on regular surfaces}

In this appendix, we study
principal curvature surfaces associated with regular surfaces.
To the best of our knowledge,
there has been no systematic study of these surfaces.
One possible reason is that, unlike the focal set (caustic),
which is compatible with homothetic scaling,
the principal curvature surface is not.

Indeed, let $f:(\R^2,0)\to(\R^3,0)$ be a regular surface and let
$\ell\in\R_{>0}$.
The principal curvatures of the scaled surface $\ell f$ are
$\kappa_i/\ell$, where $\kappa_i$ $(i=1,2)$ are the principal curvatures of $f$.
Thus the principal curvature surfaces of $\ell f$ are
$$
\ell f+\frac{\kappa_i}{\ell}\nu,
$$
whereas the caustics of $\ell f$ are
$$
\ell\left(f+\frac{\nu}{\kappa_i}\right),
$$
where $\nu$ is a unit normal vector.
As we shall see below, the singularities of principal curvature surfaces
depend on the scaling factor $\ell$.

Furthermore, the focal set is known to be the bifurcation set
of an unfolding of a function, and its singularities occur on fronts.
On the other hand, as we have seen in this paper,
principal curvature surfaces may admit Whitney umbrella singularities,
which do not occur on fronts.
Thus, the types of singularities that appear are completely different.

\subsection{Non-umbilic case}

Let $f:(\R^2,0)\to(\R^3,0)$ be a regular surface that
$0$ is not an umbilic point.
We assume that $f$ is written in Monge form
\begin{equation}\label{eq:monge}
	f=\left(u,v,\dfrac{\alpha_{20}}{2}u^2+
	\dfrac{\alpha_{02}}{2}v^2
	+\sum_{i+j\geq3}\dfrac{a_{ij}}{i!j!}u^iv^j
	\right).
\end{equation}
Since $0$ is not an umbilic point,
we may set
$\alpha_{20}=\alpha-k$, $\alpha_{02}=\alpha+k$ $(k>0)$.
Then one principal curvature $\kappa_1$ can be calculated as
\begin{align*}
	\kappa_1=&\alpha-k+\alpha_{30}u+\alpha_{21}v\\
	&\hspace{10mm}
	+\frac{1}{2k}
	\Big(\big(-\alpha_{21}^2+(-3\alpha^3+\alpha_{40})k+9\alpha^2k^2-9\alpha k^3+3k^4)\big)u^2\\
	&\hspace{10mm}
	+(-2\alpha_{12}\alpha_{21}+2\alpha_{31}k)uv\\
	&\hspace{10mm}
	+\big(-\alpha_{12}^2+k(-\alpha^3+\alpha_{22})-\alpha^2k^2+\alpha k^3+k^4)\big)v^2\Big)+O(3).
\end{align*}
Then we obtain a concrete expression of $h_{\kappa_1}$.
We set $h=h_{\kappa_1}$.  The first derivatives of $h$ are
\begin{equation}\label{eq:regdiff}
	h_u=(1 - (\alpha^{2}-k), 0, \alpha_{30}),\quad
	h_v=(0,1 - \alpha^2+k^2,\alpha_{21}).
\end{equation}
We see that $(1 - (\alpha^{2}-k),1 - \alpha^2+k^2)\ne(0,0)$.
Thus,
\begin{align*}
	&S(h)=\{k=\alpha\pm1,\ \alpha_{30}=0,\ 1 - \alpha^2+k^2\ne0\}\\
	&\hspace{20mm}\sqcup
	\left\{k=\sqrt{-1+\alpha^2},\ \alpha_{21}=0,\ 1 - (\alpha^{2}-k)\ne0\right\}.
\end{align*}
If  $k=\alpha\pm1$ and $\alpha_{30}=0$,
then $(\partial_v,\partial_u)$ is an adapted pair.
If  $k=\sqrt{-1+\alpha^2}$ and $\alpha_{21}=0$,
then $(\partial_u,\partial_v)$ is an adapted pair.
We have
\begin{align}
	\label{eq:regh2diff}
	h_{uu}&=\Big(
	3\alpha_{30}(-\alpha+k),\;
	\alpha_{21}(-\alpha+k),\\
	&\hspace{20mm}
	\alpha-4\alpha^3+\alpha_{40}-\dfrac{\alpha_{21}^2}{k}-k+12\alpha^2k-12\alpha k^2+4k^3\Big),
	\nonumber\\
	h_{uv}&=\Big(
	2\alpha_{21}(-\alpha+k),\;
	-\alpha(\alpha_{12}+\alpha_{30})+(\alpha_{12}-\alpha_{30})k,\;
	\alpha_{31}-\frac{\alpha_{12}\alpha_{21}}{k}\Big),\nonumber\\
	h_{vv}&=\Big(
	\alpha_{12}(-\alpha+k),\;
	-\alpha(\alpha_{03}+2\alpha_{21})+(\alpha_{03}-2\alpha_{21})k,\nonumber\\
	&\hspace{20mm}
	\alpha-2\alpha^3+\alpha_{22}-\dfrac{\alpha_{12}^2}{k}+k-2\alpha^2k+2\alpha k^2+2k^3
	\Big).\nonumber
\end{align}
Substituting these expressions into \eqref{eq:invs}, we get 
the invariants $W_{21}(h)$, $W_{31}(h)$ and $W_{331}(h)$ 
at an $S$-type singularity on $h$ at the origin.
However, due to their length, we omit the explicit expressions,
and present only $W_{31}(h)$ and $W_{331}(h)$
for the case $k=\alpha\pm1$ and $\alpha_{30}=0$.
We assume $k=\alpha\pm1$ and $\alpha_{30}=0$.
By \eqref{eq:regdiff} and \eqref{eq:regh2diff}, the germ of $h$ at the origin
is of $S$-type if and only if
$6-3\alpha_{21}^2+2\alpha_{40}+2\alpha(3+\alpha_{40})\ne0$.
Assuming this condition, we obtain
\begin{align*}
	W_{31}(h)&\simeq
	-8\alpha_{21}\alpha_{31}+\alpha_{12}(6+9\alpha_{21}^2+2\alpha_{40})
	+\big(-8\alpha_{21}\alpha_{31}+2\alpha_{12}(3+\alpha_{40})\big)\alpha,\\
	W_{331}(h)&\simeq
	\alpha_{21}.
\end{align*}

Furthermore, by interchanging $(u,v)$ and $(v,u)$,
replacing $k$ by $-k$,
replacing $\alpha_{ij}$ by $\alpha_{ji}$,
and applying a rotation by $\pi/2$ about the third axis,
we obtain $h_{\kappa_2}$ from $h_{\kappa_1}$.
The corresponding calculations for $h_{\kappa_2}$ are omitted.

\subsection{Umbilic case}

Let $f:(\R^2,0)\to(\R^3,0)$ be a regular surface such that
$0$ is an umbilic point.
We write $f$ as the Monge form \eqref{eq:monge} with
$\alpha_{02}=\alpha_{20}$.
We set 
$$
\sigma=-\alpha_{30}\alpha_{12}-\alpha_{21}\alpha_{03}+\alpha_{12}^{2}+\alpha_{21}^{2}.
$$
If $\sigma>0$, then the index of the configuration of the
principal directions of $f$
around $0$ is $1/2$, and
if $\sigma<0$, then the index is $-1/2$.
An umbilic point satisfying $\sigma\ne0$ is said to be
{\it non-degenerate}.
Around a non-degenerate umbilic point,
the configuration of the principal directions is of
lemon, or star,  or monstar  type in the Darboux classification.
See \cite{bf} and \cite[Section 12]{por} in details.
We consider the blown-up space $\M$ and the blow-up \eqref{eq:bl}.
We have the following theorem.
\begin{theorem}\label{thm:umbpri}
	If\/ $\sigma\ne0$, and\/ $\alpha_{20}\ne0$, then
	the principal curvatures are of class\/ $C^\infty$
	as a map\/ $(\M,\E)\to\R$.
\end{theorem}
\begin{proof}
	As we saw in the proof of Lemma \ref{lem:k1cinf},
	the principal curvatures are given in \eqref{eq:kpm}.
	Since $f$ is given as \eqref{eq:monge} with
	$\alpha_{02}=\alpha_{20}$, in the setting 
	of the theorem we are proving, we have
	\begin{align*}
		A_1=&2 \alpha_{20}+O(r),\\
		\pm\sqrt{D}=&\pm\frac{1}{2\sqrt{2}}\sqrt{r^{2}\Big(B_1 +B_2\cos2\theta+B_3 \sin2\theta\Big)+O(r^3)},\\
=&\frac{\pm|r|}{2\sqrt{2}}\sqrt{B_1 +B_2\cos2\theta+B_3 \sin2\theta+O(r)}\\
           =&
\frac{r}{2\sqrt{2}}\sqrt{B_1 +B_2\cos2\theta+B_3 \sin2\theta+O(r)}\\
           =&
\frac{r}{2\sqrt{2}}\sqrt{B_1 +(B_2^2+B_3^2)^{1/2}\cos(2\theta+\zeta)+O(r)}\\
		B_1=&
		\alpha_{03}^{2}+5 \alpha_{12}^{2}-2 \alpha_{03} \alpha_{21}
		+5 \alpha_{21}^{2}-2 \alpha_{12} \alpha_{30}+\alpha_{30}^{2},\\
		B_2=&
		-(\alpha_{03}^{2}+3 \alpha_{12}^{2}-2 \alpha_{03} \alpha_{21}-3 \alpha_{21}^{2}+2 \alpha_{12} \alpha_{30}-\alpha_{30}^{2}),\\
		B_3=&
		+2(\alpha_{03} \alpha_{12}+3 \alpha_{12} \alpha_{21}-\alpha_{03} \alpha_{30}+\alpha_{21} \alpha_{30}),
	\end{align*}
where $\zeta$ is a certain real number.
	Thus $D$ does not vanish when $r\to0$ if and only if
	$B_1^2>B_2^2+B_3^2$.
	By a direct calculation, we see
	$B_1^2-(B_2^2+B_3^2)=16\sigma^2$.
	Furthermore, since $\alpha_{20}\ne0$,
	the denominator in \eqref{eq:kpm} does not vanish near $\E$.
	Hence the principal curvatures are of class $C^\infty$
	as maps $(\M,\E)\to\R$.
\end{proof}

Under the assumptions of this theorem,
the two principal curvatures $\kappa_1$ and $\kappa_2$,
as well as their corresponding principal curvature surfaces
$h_{\kappa_i}$ $(i=1,2)$, are well-defined.
However, each $h_{\kappa_i}(\E)$ $(i=1,2)$ consists of a single point,
and the analysis of the singularities of these surfaces lies outside
the scope of this paper.


\ifrmirejected{
\medskip
{\footnotesize
	\begin{flushright}
		\begin{tabular}{l}
			(Luciana F. Martins)\\
			Departamento de Matem{\'a}tica,\\
			IBILCE - UNESP, \\
			R. Crist{\'o}v{\~a}o Colombo, 2265, \\
			S{\~a}o Jos{\'e} do Rio Preto, SP\\
			CEP 15054-000, Brazil\\
			\\
			(Kentaro Saji and Runa Shimada)\\
			Department of Mathematics,\\
			Graduate School of Science, \\
			Kobe University, \\
			Rokkodai 1-1, Nada, Kobe \\
			657-8501, Japan\\
			\\
			(Samuel P. dos Santos)\\
			Departamento de Matem{\'a}tica,\\
			Universidade Estadual de Maring{\'a}-UEM, \\
			Av. Colombro, 5790, Maring{\'a}, PR\\
			CEP 87020-900, Brazil\\
		\end{tabular}
	\end{flushright}
}
}

\begin{thebibliography}{000}
	\bibitem{bf}
	J. W. Bruce and D. L. Fidal,
	{\it On binary differential equations and umbilics},
	Proc. Roy. Soc. Edinburgh Sect. A {\bf 111} (1989), 147--168.
	
	\bibitem{bw}
	J. W. Bruce and J. M. West,
	{\it Functions on a crosscap},
	Math. Proc. Cambridge Philos. Soc. {\bf 123} (1998), no. 1, 19--39.
	
	
	\bibitem{fh-fronts}
	T. Fukui and M. Hasegawa,
	\textit{Fronts of Whitney umbrella - a differential
		geometric approach via blowing up}, 
	J. Singul. {\bf 4} (2012), 35--67.
	
	\bibitem{fh-normal}
	T. Fukui and M. Hasegawa,
	\textit{Distance squared functions on singular surfaces parameterized 
		by smooth maps\/ $\A$-equivalent
		to\/ $S_k$, $B_k$, $C_k$ and\/ $F_4$}, 
	In: Deepening and Evolution of Applied Singularity Theory,
	Adv. Stud. Pure Math. {\bf 89}, 113--151, Math. Soc. Japan, 2025.
	
	\bibitem{hhnuy}
	M. Hasegawa, A. Honda, K. Naokawa, M. Umehara and K. Yamada, 
	{\it Intrinsic
		invariants of cross caps},
	Sel. Math. New Ser. {\bf 20} (2014), 769--785.
	
	\bibitem{mn}
	L. F. Martins and J. J. Nu\~no-Ballesteros,
	\textit{Contact properties of surfaces in\/ 
		$\R^3$ with corank\/ $1$ singularities},
	Tohoku Math. J. \bf{67} \rm{(2015), 105--124.}
	
	\bibitem{mond} D. Mond,
	\textit{On the classification of germs of maps from\/ $\R^2$ to\/ $\R^3$},
	Proc. London Math. Soc. \bf{50} \rm{(1985), 333-369.}
	
	\bibitem{por}
	I. R. Porteous,
	{\it Geometric differentiation. 
		For the intelligence of curves and surfaces},
	Second edition. Cambridge University Press, Cambridge, 2001.
	
	\bibitem{shimadas1}
	R. Shimada,
	{\it Geometry on deformations of\/ $S_1$ singularities},
	Houston J. Math.,
	{\bf 50} (2024), 873--890.
	
	\bibitem{shimada2para}
	R. Shimada,
	{\it Geometric deformations of codimension two singularities},
	preprint.
	
	\bibitem{tari}
	F. Tari,
	{\it On pairs of geometric foliations on a cross-cap},
	Tohoku Math. J. (2) {\bf 59} (2007), no. 2, 233--258. 
	
	\bibitem{west}
	J. M. West,
	{\it The differential geometry of the cross-cap},
	PhD thesis, University of Liverpool (1995).
	
\end{thebibliography}
\end{document}

E-mail: {\tt saji@math.kobe-u.ac.jp}\\